\documentclass{amsart}

\usepackage{amsmath,amssymb}

\usepackage[all]{xy}

\def\spec#1#2{\medskip\noindent\textbf{#1.} \emph{#2}}

\newtheorem{theorem}{Theorem}[section] 
\newtheorem{lemma}[theorem]{Lemma}     
\newtheorem{corollary}[theorem]{Corollary}
\newtheorem{proposition}[theorem]{Proposition}
\theoremstyle{definition}
\newtheorem{definition}[theorem]{Definition}
\newtheorem{remark}[theorem]{Remark}
\newtheorem*{remark*}{Remark} 
\newtheorem*{question}{Question} 

\numberwithin{equation}{section}

\title[Poincar{\'e} duality complexes in dimension four]
{Poincar{\'e} duality complexes in dimension four} 

\author{Hans Joachim Baues and Beatrice Bleile}
 
\begin{document}

\begin{abstract}

We describe an algebraic structure on chain complexes yielding algebraic models which classify homotopy types of ${\rm PD}^4$--complexes. Generalizing Turaev's fundamental triples of ${\rm PD}^3$--complexes we introduce fundamental triples for ${\rm PD}^n$--complexes and show that two ${\rm PD}^n$--complexes are orientedly homotopy equivalent if and only if their fundamental triples are isomorphic. As applications we establish a conjecture of Turaev and obtain a criterion for the existence of degree $1$ maps between $n$--dimensional manifolds.

\end{abstract}

\maketitle

\section*{Introduction}
In order to study the homotopy types of closed manifolds, Browder and Wall introduced the notion of Poincar{\'e} duality complexes. A Poincar{\'e} duality complex, or ${\rm PD}^n$--complex, is a ${\rm CW}$--complex, $X$, whose cohomology satisfies a certain algebraic condition. Equivalently, the chain complex, $\widehat C(X)$, of the universal cover of $X$ must satisfy a corresponding algebraic condition. Thus Poincar{\'e} complexes form a mixture of topological and algebraic data and it is an old quest to provide purely algebraic data determining the homotopy type of ${\rm PD}^n$--complexes. This has been achieved for $n=3$, but, for $n=4$, only partial results are available in the literature.

Homotopy types of $3$--manifolds and ${\rm PD}^3$--complexes were considered by Thomas \cite{Thomas}, Swarup \cite{Swarup} and Hendriks \cite{Hendriks}. The homotopy type of a ${\rm PD}^3$--complex, $X$, is determined by its \emph{fundamental triple}, consisting of the fundamental group, $\pi = \pi_1(X)$, the orientation character, $\omega$, and the image in ${\rm H}_3(\pi, \mathbb Z^{\omega})$ of the fundamental class, $[X]$. Turaev \cite{Turaev} provided an algebraic condition for a triple to be realizable by a ${\rm PD}^3$--complex. Thus, in dimension $3$, there are purely algebraic invariants which provide a complete classification.

Using primary cohomological invariants like the fundamental group, characteristic classes and intersection pairings, partial results were obtained for $n=4$ by imposing conditions on the fundamental group. For example, Hambleton, Kreck and Teichner classified ${\rm PD}^4$--complexes with finite fundamental group having periodic cohomology of dimension $4$ (see \cite{HamblKreck}, \cite{Teichner} and \cite{HamblKreckTeich}). Cavicchioli, Hegenbarth and Piccarreta studied ${\rm PD}^4$--complexes with free fundamental group (see \cite{CavHeg} and \cite{HegPic}), as did Hillman \cite{Hillman3}, who also considered ${\rm PD}^4$--complexes with fundamental group a ${\rm PD}^3$--group \cite{Hillman2}. Recently, Hillman \cite{Hillman1} considered homotopy types of ${\rm PD}^4$--complexes whose fundamental group has cohomological dimension $2$ and one end. 

It is doubtful whether primary invariants are sufficient for the homotopy classification of ${\rm PD}^4$--complexes in general and we thus follow Ranicki's approach (\cite{Ranicki1} and \cite{Ranicki2}) who assigned to each ${\rm PD}^n$--complex, $X$, an \emph{algebraic Poincar{\'e} duality complex} given by the chain complex, $\widehat C(X)$, together with a \emph{symmetric} or \emph{quadratic structure}. However, Ranicki considered neither the realizability of such algebraic Poincar{\'e} duality complexes nor whether the homotopy type of a ${\rm PD}^n$--complex is determined by the homotopy type of its algebraic Poincar{\'e} duality complex.

This paper presents a structure on chain complexes which completely classifies ${\rm PD}^4$--complexes up to homotopy. The classification uses \emph{fundamental triples} of ${\rm PD}^4$--complexes, and, in fact, the chain complex model yields algebraic conditions for the realizability of fundamental triples.

Fundamental triples  of formal dimension $n \geq 3$ comprise an $(n-2)$--type $T$, a homomorphism $\omega: \pi_1(T) \rightarrow \mathbb Z / 2 \mathbb Z$ and a homology class $t \in {\rm H}_n(T, \mathbb Z^{\omega})$. There is a functor, 
\begin{equation*}
\tau_+: {\rm \bf PD}^n_+ \longrightarrow {\rm\bf Trp}^n_+,
\end{equation*}
from the category ${\rm \bf PD}^n_+$ of ${\rm PD}^n$--complexes and maps of degree one to the category ${\rm\bf Trp}^n_+$ of triples and morphisms inducing surjections on fundamental groups. Our first main result is

\spec{Theorem \ref{mainthm}}{The functor $\tau_+$ reflects isomorphisms and is full  for $n \geq 3$.}

\spec{Corollary \ref{class}}{Take $n \geq 3$. Two closed $n$--dimensional manifolds or two ${\rm PD}^n$--complexes, respectively, are orientedly homotopy equivalent if and only if their fundamental triples are isomorphic.}

\medskip

Corollary \ref{class} extends results of Thomas \cite{Thomas}, Swarup \cite{Swarup} and Hendriks \cite{Hendriks} for dimension $3$ to arbitrary dimension and establishes Turaev's conjecture \cite{Turaev} on ${\rm PD}^n$--complexes whose $(n-2)$--type is an Eilenberg--{Mac\,Lane} space $K(\pi_1X,1)$. Corollary \ref{class} is even of interest in the case of simply connected or highly connected manifolds. 

Theorem \ref{mainthm} also yields a criterion for the existence of a map of degree one between ${\rm PD}^n$--complexes, recovering Swarup's result for maps between $3$--manifolds and Hendriks' result for maps between ${\rm PD}^3$--complexes. 

In the oriented case, special cases of Corollary \ref{class} were proved by Hambleton and Kreck \cite{HamblKreck} and Cavicchioli and Spaggiari \cite{CavSpag}. In fact, in \cite{HamblKreck}, Corollary 3.2 is obtained under the condition that either the fundamental group is finite or the second rational homology of the 2--type is non--zero. Corresponding conditions were used in \cite{CavSpag} for oriented ${\rm PD}^{2n}$--complexes with $(n-1)$--connected universal covers, and Teichner extended the approach of \cite{HamblKreck} to the non--oriented case in his thesis \cite{Teichner}. Our result shows that the conditions on finiteness and rational homology used in these papers are not necessary.  

It follows directly from Poincar{\'e} duality and Whitehead's Theorem that the functor $\tau_+$ reflects isomorphisms. To show that $\tau_+$ is full requires work. Given ${\rm PD}^n$--complexes $Y$ and $X$, $n \geq 3$, and a morphism $f: \tau_+ Y \rightarrow \tau_+ X$ in ${\rm\bf Trp}^n_+$, we first construct a chain map $\xi: \widehat C(Y) \rightarrow \widehat C(X)$  preserving fundamental classes, that is, $\xi_{\ast}[Y] = [X]$. Then we use the category ${\rm\bf H}^{k+1}_c$ of homotopy systems of order $(k+1)$ introduced in \cite{Baues1} to realize $\xi$ by a map $\overline f: Y \rightarrow X$ with $\tau_+(\overline f) = f$. 

Our second main result describes algebraic models of homotopy types of ${\rm PD}^4$--complexes. We introduce the notion of \emph{${\rm PD}^n$--chain complex} and show that ${\rm PD}^3$--chain complexes are equivalent to ${\rm PD}^3$--complexes up to homotopy. In Section \ref{4chainrealize} we show that ${\rm PD}^4$--chain complexes classify homotopy types of ${\rm PD}^4$--complexes up to $2$--torsion. In particular, we obtain

\spec{Theorem \ref{finodd}}{The functor $\widehat C$ induces a 1--1 correspondence between homotopy types of ${\rm PD}^4$--complexes with finite fundamental group of odd order and homotoppy types of ${\rm PD}^4$--chain complexes with homotopy co--commutative diagonal and finite fundamental group of odd order.}

\medskip

To obtain a complete homotopy classification of ${\rm PD}^4$--complexes, we study the chain complex of a $2$--type in Section \ref{2typechains}. We compute this chain complex  up to dimension $4$  in terms of Peiffer commutators in pre--crossed modules. This allows us to introduce ${\rm PD}^4$--chain complexes together with a $\beta$--invariant, and we prove

\spec{Corollary \ref{PD4cor}}{The functor $\widehat C$ induces a 1--1 correspondence between homotopy types of ${\rm PD}^4$--complexes and homotopy types of $\beta$--${\rm PD}^4$--chain complexes.}

\medskip

Corollary \ref{PD4cor} highlights the crucial r\^{o}le of Peiffer commutators for the homotopy classification of $4$--manifolds.

The proofs of our results rely on the obstruction theory \cite{Baues1} for the realizability of chain maps which we recall in Section \ref{homsys}.
\section{Chain complexes}\label{chaincomp}
Let $X^n$ denote the $n$--skeleton of the CW--complex $X$. We call $X$ reduced if $X^0 = \ast$ is the base point. The objects of the category ${\rm\bf{CW}}_0$ are reduced CW--complexes $X$ with universal covering $p: \widehat X \rightarrow X$, such that $p(\widehat  \ast) = \ast$, where $\widehat \ast \in  \widehat X^0$ is the base point of $\widehat X$. Here the $n$--skeleton of $\widehat X$ is $\widehat X^n = p^{-1}(X^n)$. Morphisms in ${\rm\bf{CW}}_0$ are cellular maps $f: X \rightarrow Y$ and homotopies in $\rm\bf CW_0$ are base point preserving. A map $f: X \rightarrow Y$ in $\rm\bf CW_0$ induces a unique covering map $\widehat f: \widehat X \rightarrow \widehat Y$ with $\widehat f(\widehat \ast) = \widehat \ast$, which is equivariant with respect to $\varphi = \pi_1(f)$. 

We consider pairs $(\pi, C)$, where $\pi$ is a group and $C$ a chain complex of left modules over the group ring $\mathbb Z [\pi]$. We write $\Lambda = \mathbb Z [\pi]$ and $C$ for $(\pi, C)$, whenever $\pi$ is understood. We call $(\pi, C)$ free if each $C_n, n \in \mathbb Z$, is a free $\Lambda$--module. Let $\rm aug: \Lambda \rightarrow \mathbb Z$ be the augmentation homomorphism, defined by ${\rm aug} (g) = 1$ for all $g \in \pi$. Every group homomorphism, $\varphi: \pi \rightarrow \pi'$, induces a ring homomorphism $\varphi_{\sharp}: \Lambda \rightarrow \Lambda'$, where $\Lambda' = \mathbb Z[\pi']$. A chain map is a pair $(\varphi, F): (\pi, C) \rightarrow (\pi', C')$, where $\varphi$ is a group homomorphism and $F: C \rightarrow C'$ a $\varphi$--equivariant chain map, that is a chain map of the underlying abelian chain complexes, such that $F(\lambda c) = \varphi_{\sharp}(\lambda) F(c)$ for $\lambda \in \Lambda$ and $c \in C$. Two such chain maps are homotopic, $(\varphi, F) \simeq (\psi, G)$ if $\varphi = \psi$ and if there is a $\varphi$--equivariant map $\alpha: C \rightarrow C'$ of degree $+1$ such that $G - F = d \alpha + \alpha d$.

A pair $(\pi, C)$ is a reduced chain complex if $C_0 = \Lambda$ with generator $\ast$, $C_i = 0$ for $i  < 0$ and ${\rm H}_0 C = \mathbb Z$ such that $C_0 = \Lambda \rightarrow {\rm H}_0 C = \mathbb Z$ is the augmentation of $\Lambda$. A chain map, $(\varphi, f): (\pi, C) \rightarrow (\pi', C')$, of reduced chain complexes, is reduced if $f_0$ is induced by $\varphi_{\sharp}$, and a chain homotopy $\alpha$ of reduced chain maps is reduced if $\alpha_0 = 0$. The objects of the category ${\rm\bf H}_0$ are reduced chain complexes and the morphisms are reduced chain maps. Homotopies in ${\rm\bf H}_0$ are reduced chain homotopies. Every chain complex $(\pi, C)$ in ${\rm\bf H}_0$ is equipped with the augmentation $\varepsilon: C \rightarrow \mathbb Z$ in ${\rm\bf H}_0$. The ring homomorphism $\mathbb Z \rightarrow \Lambda$ yields the co--augmentation $\iota: \mathbb Z \rightarrow C$, where we view $\mathbb Z = (0, \mathbb Z)$ as chain complex with trivial group $\pi = 0$ concentrated in degree $0$. Note that $\varepsilon \iota = {\rm id}_{\mathbb Z}$, and the composite $\iota \varepsilon: C \rightarrow C'$ is the \emph{trivial} map.

For an object $X$ in ${\rm\bf{CW}}_0$, the cellular chain complex $C(\widehat X)$ of the universal cover $\widehat X$ is given by $C_n(\widehat X) = {\rm H}_n(\widehat X^n, \widehat X^{n+1})$. The fundamental group $\pi = \pi_1(X)$ acts on $C(\widehat X)$,  and viewing $C(\widehat X)$ as a complex of left $\Lambda$--modules, we obtain the object $\widehat C(X) = (\pi, C(\widehat X))$ in ${\rm\bf H}_0$. Moreover, a morphism $f: X \rightarrow Y$ in ${\rm\bf CW}_0$ induces the homomorphism $\pi_1(f)$ on the fundamental groups and the $\pi_1(f)$--equivariant map $\widehat f: \widehat X \rightarrow \widehat Y$ which in turn induces the $\pi_1(f)$--equivariant chain map $\widehat f_{\ast}: C(\widehat X) \rightarrow C(\widehat Y)$ in ${\rm\bf H}_0$. As $\widehat f$ preserves base points, $\widehat C(f) = (\pi_1(f), \widehat f_{\ast})$ is a reduced chain map. We obtain the functor
\begin{equation}\label{Chat}
\widehat C: {\rm\bf CW}_0  \longrightarrow {\rm\bf H}_0.
\end{equation}
The chain complex $C$ in ${\rm\bf H}_0$ is \emph{$2$--realizable} if there is an object $X$ in ${\rm\bf CW}_0$ such that $\widehat C(X^2) \cong C_{\leq 2}$, that is, $\widehat C(X^2)$ is isomorphic to $C$ in degree $\leq 2$. 
Given two objects $X$ and $Y$ in ${\rm\bf{CW}}_0$, their product again carries a cellular structure and we obtain the object $X \times Y$ in ${\rm\bf{CW}}_0$ with base point $(\ast, \ast)$ and universal cover $(X \times Y)\widehat \  = \widehat X \times \widehat Y$, so that 
\begin{equation}
\widehat C(X \times Y) = (\pi \times \pi, C(\widehat X) \otimes_{\mathbb Z} C(\widehat Y)).
\end{equation}

For $i = 1, 2$, let $p_i: X \times X \rightarrow X$ be the projection onto the $i$--th factor. A \emph{diagonal} $\Delta: X \rightarrow X \times X$ in ${\rm\bf{CW}}_0$ is a cellular map with $p_i \Delta \simeq {\rm{id}}_X$ in ${\rm\bf{CW}}_0$ for $i = 1, 2$. A \emph{diagonal} on $(\pi, C)$ in ${\rm\bf H}_0$ is a chain map $(\delta, \Delta): (\pi, C) \rightarrow (\pi \times \pi, C \otimes_{\mathbb Z} C)$ in ${\rm\bf H}_0$ with $\delta: \pi \rightarrow \pi \times \pi, g \mapsto (g,g)$, such that $p_i \Delta \simeq {\rm id}_C$ for $i = 1, 2$, where $p_1 = {\rm id} \otimes \varepsilon$ and $p_2 = \varepsilon \otimes {\rm id}$. 

The diagonal $(\delta, \Delta)$ in ${\rm\bf H}_0$ is homotopy co--associative if the diagram
\begin{equation*}
\xymatrix{
C \ar[rr]^-{\Delta} \ar[d]^{\Delta} && C \otimes _{\mathbb Z} C \ar[d]^-{\mbox{id} \otimes \Delta} \\
C \otimes _{\mathbb Z} C \ar[rr]^-{\Delta \otimes \mbox{id}} && C \otimes _{\mathbb Z} C \otimes _{\mathbb Z} C}
\end{equation*}
commutes up to chain homotopy in ${\rm\bf H}_0$. The diagonal $(\delta, \Delta)$ in ${\rm\bf H}_0$ is homotopy co--commutative if the diagram

\begin{equation*}
\xymatrix{
C \ar[rr]^-{\Delta} \ar[drr]_{\Delta} && C \otimes _{\mathbb Z} C \ar[d]^-{T} \\
&& C \otimes _{\mathbb Z} C}
\end{equation*}
commutes up to chain homotopy in ${\rm\bf H}_0$, where $T$ is given by $T(c \otimes d) = (-1)^{|c||d|} d \otimes c.$ 

By the cellular approximation theorem, there is a diagonal $\Delta: X \rightarrow X \times X$ in ${\rm\bf{CW}}_0$ for every object $X$ in ${\rm\bf{CW}}_0$. Applying the functor $\widehat C$ to such a diagonal, we obtain the diagonal $\widehat C(\Delta)$ in ${\rm\bf H}_0$. This raises the question of realizabilty, that is, given a diagonal $(\delta, \Delta): \widehat C (X) \rightarrow \widehat C(X) \otimes_{\mathbb Z} \widehat C(X)$ in ${\rm\bf H}_0$, is there a diagonal $\Delta$ in ${\rm\bf{CW}}_0$ with $\widehat C(\Delta) = (\delta, \Delta)$? As $\widehat C(\Delta)$ is homotopy co--associative and homotopy co--commutative for any diagonal $\Delta$ in ${\rm\bf{CW}}_0$, homotopy co--associativity and homotopy co--commutativity of $(\delta, \Delta)$ are necessary conditions for realizability. 

To discuss questions of realizability for a functor $\lambda: {\rm\bf A} \rightarrow {\rm\bf B}$, we consider pairs $(A,b)$, where $b: \lambda A \cong B$ is an equivalence in $\rm\bf B$. Two such pairs are equivalent, $(A,b) \sim (A',b')$, if and only if there is an equivalence $g: A' \cong A$ in $\rm\bf A$ with $\lambda g = b^{-1} b'$. The classes of this equivalence relation form the class of $\lambda$--realizations of $B$,
\begin{equation}
{\rm Real}_{\lambda}(B) = \{ (A,b) \ | \ b: \lambda A \cong B \} / \sim.
\end{equation}
We say that $B$ is $\lambda$--realizable if ${\rm Real}_{\lambda}(B)$ is non--empty. The functor $\lambda: {\rm\bf A} \rightarrow {\rm\bf B}$ is \emph{representative} if all objects $B$ in $\rm\bf B$ are $\lambda$--realizable. Further, we say that $\lambda$ \emph{reflects isomorphisms}, if a morphism $f$ in $\rm\bf A$ is an equivalence whenever $\lambda(f)$ is an equivalence in $\rm\bf B$. The functor $\lambda$ is \emph{full} if, for every morphism $\overline f: \lambda(A) \rightarrow \lambda(A')$ in $\rm\bf B$, there is a morphism $f: A \rightarrow A'$ in $\rm\bf A$, such that $\lambda(f) = \overline f$, we then say $\overline f$ is \emph{$\lambda$--realizable}.

\section{$\rm PD$--chain complexes and $\rm PD$--complexes}\label{PDcomp}
We start this section with a description of the cap product on chain complexes. We fix a homomorphism $\omega: \pi \rightarrow {\mathbb{Z}} / 2 {\mathbb{Z}} = \{0,1\}$ which gives rise to the anti--isomorphism $\overline{\phantom{x}}: {{\Lambda}} \rightarrow {{\Lambda}}$ of rings defined by $\overline g = (-1)^{\omega(g)} g^{-1}$ for $g \in \pi$. To the left $\Lambda$--module $M$ we associate the right $\Lambda$--module $M^{\omega}$ with the same underlying abelian group and action given by $\lambda.a = a.\overline{\lambda}$ for $a \in A$ and $\lambda \in \Lambda$. Proceeding analogously for a right $\Lambda$--module $N$, we obtain a left $\Lambda$--module $^{\omega}N$. We put
\begin{equation*}
{\rm H}_n(C,M^{\omega}) = {\rm H}_n(M^{\omega} \otimes_{\Lambda} C); \quad
{\rm H}^k(C,M) = {\rm H}_{-k}(\mbox{Hom}_{\Lambda}(C,M)).\]
To define the $\omega$--twisted cap product $\cap$ for a chain complex $C$ in ${\rm\bf H}_0$ with diagonal $(\delta, \Delta)$, write $\Delta(c) = \sum_{i+j=n, \alpha} c_{i, \alpha}' \otimes c_{j, \alpha}''$ for $c \in C$. Then 
\begin{eqnarray*}
\cap: \mbox{Hom}_{\Lambda}(C, M)_{-k} \otimes_{\mathbb Z}(\mathbb Z^{\omega} \otimes_{\Lambda} C)_{n} & \rightarrow & (M^{\omega} \otimes_{\Lambda} C)_{n-k} \\
\psi \otimes(z \otimes c) & \mapsto & \sum_{\alpha} z \psi(c'_{k, \alpha}) \otimes c''_{n-k, \alpha}
\end{eqnarray*}
for every left $\Lambda$--module $M$. Passing to homology and composing with 
\begin{gather*}
{\rm H}^{\ast} (C, M) \otimes_{\mathbb Z} {\rm H}_{\ast} (C \otimes_{\mathbb Z} C, \mathbb Z^{\omega}) \rightarrow {\rm H}_{\ast} (\mbox{Hom}_{\Lambda}(C, M)\big) \otimes_{\mathbb Z} \big( \mathbb Z^{\omega} \otimes_{\Lambda} (C \otimes_{\mathbb Z} C) )), \\
[\psi] {\otimes}[y] \mapsto [\psi {\otimes}y],
\end{gather*}
we obtain
\begin{equation}
\cap: {\rm H}^k(C, M) \otimes_{\mathbb Z}{\rm H}_{n}(C,\mathbb Z^{\omega}) \rightarrow {\rm H}_{n-k}(C, M^{\omega}). 
\end{equation}
A \emph{$\rm PD^n$--chain complex} $C = ((\pi, C), \omega, [C], \Delta)$ consists of a free chain complex $(\pi, C)$ in ${\rm\bf H}_0$ with $\pi$ finitely presented, a group homomorphism $\omega: \pi \rightarrow \mathbb Z / 2 \mathbb Z$, a fundamental class $[C] \in {\rm H}_n(C, \mathbb Z^{\omega})$ and a diagonal $\Delta: C \rightarrow C \otimes C$ in ${\rm\bf H}_0$, such that ${\rm H}_1C = 0$ and
\begin{equation}\label{capwithfund}
\cap [C]:   {\rm H}^{r}(C, M) \rightarrow {\rm H}_{n-r}(C, M^{\omega});
\quad \alpha \mapsto \alpha \cap [X]
\end{equation}
is an isomorphism of abelian groups for every $r \in \mathbb Z$ and every left $\Lambda$--module $M$. A morphism of $\rm PD^n$--chain complexes $f: ((\pi, C), \omega, [C], \Delta) \rightarrow ((\pi', C'), \omega', [C'], \Delta')$ is a morphism $(\varphi, f): (\pi, C) \rightarrow (\pi', C')$ in ${\rm\bf H}_0$ such that $\omega = \omega' \varphi$ and $(f \otimes f) \Delta \simeq \Delta' f$. The category ${\rm\bf PD}^n_{\ast}$ is the category of $\rm PD^n$--chain complexes and morphisms between them. Homotopies in ${\rm\bf PD}^n_{\ast}$ are reduced chain homotopies. The subcategory ${\rm\bf PD}^n_{\ast+}$ of ${\rm\bf PD}^n_{\ast}$ is the category consisting of ${\rm PD}^n$--chain complexes and \emph{oriented} or \emph{degree $1$} morphisms of ${\rm PD}^n$--chain complexes, that is, morphisms $f: C \rightarrow D$ with $f_{\ast}[C] = [D]$.
Wall \cite{Wall1} showed that it is enough to demand that (\ref{capwithfund}) is an isomorphism for $M = \Lambda$. If $1 \otimes x \in \mathbb Z^{\omega } \otimes_{\Lambda} C_n$ represents the fundamental class $[C]$, where $C_i$ is finitely generated for $i \in \mathbb Z$, then $\cap [C]$ in (\ref{capwithfund}) is an isomorphism if and only if 
\begin{equation}
\cap 1 \otimes x: C^{\ast} = \ ^{\omega}{\rm Hom}_{\Lambda} (C, ^{\omega}\!\!\Lambda) 
\rightarrow \Lambda \otimes_{\Lambda} C = C
\end{equation}
is a homotopy equivalence of chain complexes of degree $n$. Here finite generation implies that $C^{\ast}$ is a free chain complex.
\begin{lemma}\label{2real}
Every ${\rm PD}^n$--chain complex is homotopy equivalent in ${\rm\bf PD}^n_{\ast}$ to a $2$--realizable ${\rm PD}^n$--chain complex.
\end{lemma}
\begin{proof}
This follows from Proposition III 2.13 and Theorem III 2.12 in \cite{Baues1}.
\end{proof}
A \emph{$\rm PD^n$--complex} $X = (X, \omega, [X], \Delta)$ consists of an object $X$ in ${\rm\bf CW}_0$ with finitely presented fundamental group $\pi_1(X)$, a group homomorphism $\omega: \pi_1 X \rightarrow \mathbb Z / 2 \mathbb Z$, a fundamental class $[X] \in {\rm H}_n(X, \mathbb Z^{\omega})$ and a diagonal $\Delta: X \rightarrow X \times X$ in ${\rm\bf CW}_0$, such that $(\widehat CX, \omega, [X], \widehat C \Delta)$ is a $\rm PD^n$--chain complex. A morphism of $\rm PD^n$--complexes $f: (X, \omega, [X], \Delta) \rightarrow (X', \omega', [X'], \Delta')$ is a morphism $f: X \rightarrow X'$ in ${\rm\bf CW}_0$ such that $\omega = \omega' \pi_1(f)$. The category ${\rm\bf PD}^n$ is the category of $\rm PD^n$--complexes and morphisms between them. Homotopies in ${\rm\bf PD}^n$ are homotopies in ${\rm\bf CW}_0$. The subcategory ${\rm\bf PD}^n_+$ of ${\rm\bf PD}^n$ is the category consisting of ${\rm PD}^n$--complexes and \emph{oriented} or \emph{degree $1$} morphisms of ${\rm PD}^n$--complexes, that is, morphisms $f: X \rightarrow Y$ with $f_{\ast}[X] = [Y]$.

\begin{remark}
Our ${\rm PD}^n$--complexes have finitely presented fundamental groups by definition and are thus finitely dominated by Propostion 1.1 in \cite{Wall3}.
\end{remark}

Let $X$ be a ${\rm PD}^n$--complex with $n \geq 3$. We say that $X$ is \emph{standard}, if $X$ is a $\rm CW$--complex which is $n$--dimensional and has exactly one $n$--cell $e^n$. We say that $X$ is \emph{weakly standard}, if $X$ has a subcomplex $X'$ with $X = X' \cup e^n$, where $X'$ is $n$--dimensional and satisfies ${\rm H}^n(X', B) = 0$ for all coefficient modules $B$. In this sense $X'$ is homologically $(n-1)$--dimensional. Of course standard implies weakly standard with $X' = X^{n-1}$.

\begin{remark*}
Every compact connected manifold $M$ of dimension $n$ has the homotopy type of a finite standard $\rm PD^n$--complex. 
\end{remark*}

\begin{remark}\label{Wallres}
Wall's Theorem 2.4 in \cite{Wall1} and Theorem E in \cite{Wall2} imply that, for $n \geq 4$, every ${\rm PD}^n$--complex is homotopy equivalent to a standard ${\rm PD}^n$--complex and, for $n = 3$, every ${\rm PD}^3$--complex is homotopy equivalent to a weakly standard ${\rm PD}^3$--complex.
\end{remark}

Let $C$ be a ${\rm PD}^n$--chain complex with $n \geq 3$. We say that $C$ is \emph{standard}, if $C$ is $2$--realizable, $C_i = 0$ for $i > n$, and $C_n = \Lambda[e_n]$, where $[e_n] \in C_n$. We say that $C$ is \emph{weakly standard}, if $C$ is $2$--realizable and has a subcomplex $C'$ with $C = C' \oplus \Lambda[e_n]$, where $C'$ is $n$--dimensional and satisfies ${\rm H}^n(C', B) = 0$ for all coefficient modules $B$. 

\begin{remark}\label{corr}
A ${\rm PD}^n$--complex, $X$, is homotopy equivalent to a finite standard, standard or weakly standard ${\rm PD}^n$--complex, respectively, if and only if the ${\rm PD}^n$--chain complex $\widehat CX$ is homotopy equivalent to a finite standard, standard  or weakly standard ${\rm PD}^n$--chain complex, respectively.
\end{remark}

\section{Fundamental triples}\label{triplesection}
Homotopy types of $3$--manifolds and ${\rm PD}^3$--complexes were considered by Thomas \cite{Thomas}, Swarup \cite{Swarup} and Hendriks \cite{Hendriks}. In particular, Hendriks and Swarup provided a criterion for the existence of degree $1$ maps between $3$--manifolds and ${\rm PD}^3$--complexes, respectively. In this section we generalize these results to manifolds and Poincar{\'e} duality complexes of arbitrary dimension $n$.

Let $k$--{\bf types} be the full subcategory of ${\rm\bf CW}_0 /\simeq$ consisting of $\rm CW$--complexes $X$ in ${\rm\bf CW}_0$ with $\pi_i(X) = 0$ for $i > k$. The $k$--th \emph{Postnikov functor}
\begin{equation*}
P_k: {\rm\bf CW}_0 \rightarrow k{\rm\bf-types}
\end{equation*}
is defined as follows. For $X$ in ${\rm\bf CW}_0$ we obtain $P_kX$ by ``killing homotopy groups'', that is, we choose a $\rm CW$--complex $P_kX$ with $(k+1)$--skeleton $(P_kX)^{k+1} = X^{k+1}$ and $\pi_i(P_kX) = 0$ for $i > k$. For a morphism $f: X \rightarrow Y$ in ${\rm\bf CW}_0$ we may choose a map $Pf: P_kX \rightarrow P_kY$ which extends the restriction $f^{k+1}: X^{k+1} \rightarrow Y^{k+1}$ as $\pi_i(P_kY) = 0$ for $i > k$. Then the functor $P_k$ assigns $P_kX$ to $X$ and the homotopy class of $Pf$ to $f$. Different choices for $P_kX$ yield canonically isomorphic functors $P_k$. The $\rm CW$--complex $P_1X = K(\pi_1X,1)$ is an Eilenberg--{Mac\,Lane} space and, as a functor, $P_1$ is equivalent to the functor $\pi_1$ of fundamental groups. There are natural maps
\begin{equation}\label{postnikov}
p_k: X \longrightarrow P_kX
\end{equation}
in ${\rm\bf CW}_0/\simeq$ extending the inclusion $X^{k+1} \subseteq P_kX$. 

For $n \geq 3$, a \emph{fundamental triple} $T = (X, \omega, t)$ \emph{of formal dimension $n$} consists of an $(n-2)$--type $X$, a homomorphism $\omega: \pi_1X \rightarrow \mathbb Z / 2 \mathbb Z$ and an element $t \in {\rm H}_n(X, \mathbb Z^{\omega})$. A morphism $(X, \omega_X, t_X) \rightarrow (Y, \omega_Y, t_Y)$ between  fundamental triples is a homotopy class $\{f\}: X \rightarrow Y$ of maps of the $(n-2)$--types, such that $\omega_X = \omega_Y \pi_1(f)$ and $f_{\ast}(t_X) = t_Y$. We obtain the category ${\rm\bf Trp}^n$ of fundamental triples $T$ of formal dimension $n$ and the functor
\begin{equation*}
\tau: {\rm\bf PD}^n_+ /\simeq \ \longrightarrow {\rm\bf Trp}^n, \quad X \longmapsto (P_{n-2}X, \omega_X, p_{n-2 \ast}[X]).
\end{equation*}
Every degree $1$ morphism $Y \rightarrow X$ in ${\rm\bf PD}^n_+$ induces a surjection $\pi_1Y \rightarrow \pi_1X$ on fundamental groups, see for example \cite{Browder1}, and hence we introduce the subcategory ${\rm\bf Trp}^n_+ \subset {\rm\bf Trp}^n$ consisting of all morphisms inducing surjections on fundamental groups. Then the functor $\tau$ yields the functor
\begin{equation}\label{tauplus}
\tau_+: {\rm\bf PD}^n_+ /\simeq \ \longrightarrow {\rm\bf Trp}^n_+.
\end{equation}
As a main result in this section we show
\begin{theorem}\label{mainthm}
The functor $\tau_+$ reflects isomorphisms and is full  for $n \geq 3$.
\end{theorem}
As corollaries we mention
\begin{corollary}\label{class}
Take $n \geq 3$. Two $n$--dimensional manifolds, respectively, two ${\rm PD}^n$--complexes, are orientedly homotopy equivalent if and only if their fundamental triples are isomorphic.
\end{corollary}
\begin{remark*}
For $n = 3$, Corollary \ref{class} yields the results by Thomas \cite{Thomas}, Swarup \cite{Swarup} and Hendriks \cite{Hendriks}. Turaev reproves Hendriks' result in the appendix of \cite{Turaev}, although the proof needs further explanation. We reprove the result again in a more algebraic way. 
\end{remark*}
\begin{remark*}
Turaev conjectures in \cite{Turaev} that his proof for $n = 3$ has a generalization to ${\rm PD}^n$--complexes whose $(n-2)$--type is an Eilenberg--{Mac\,Lane} space $K(\pi,1)$. Corollary \ref{class} proves this conjecture.
\end{remark*}
Next consider ${\rm PD}^n$--complexes $X$ and $Y$ and a diagram
\begin{equation}\label{overfex}
\xymatrix{
Y \ar[r]^-{p_{n-2}} \ar@{..>}[d]_{\overline f} & P_{n-2}Y \ar[d]^f \\
X \ar[r]^-{p_{n-2}} & P_{n-2}X.}
\end{equation}
\begin{corollary}\label{degoneex}
For $n \geq 3$, there is a degree $1$ map $\overline f$ rendering Diagram (\ref{overfex}) homotopy commutative if and only if $f$ induces a surjection on fundamental groups, is compatible with the orientations $\omega_X$ and $\omega_Y$, that is, $\omega_X \pi_1(f) = \omega_Y$, and
\begin{equation*}
f_{\ast}p_{n-2 \ast}[Y] = p_{n-2 \ast}[X].
\end{equation*}
\end{corollary}
\begin{remark*}
Swarup \cite{Swarup} and Hendriks \cite{Hendriks} prove Corollary \ref{degoneex} for $3$--manifolds and ${\rm PD}^3$--complexes, respectively.
\end{remark*}
\begin{remark*}
For oriented ${\rm PD}^4$--complexes with finite fundamental group and $f$ a homotopy equivalence, the map $\overline f$ corresponds to the map $h$ in Lemma 1.3 \cite{HamblKreck} of Hambleton and Kreck. The reader is invited to compare our proof to that of Lemma 1.3 \cite{HamblKreck} which shows the existence of $h$ but not the fact that $h$ is of degree 1.
\end{remark*}
By Remark \ref{Wallres}, Theorem \ref{mainthm} is a consequence of the following Lemmata \ref{tauplusrefl} and \ref{mainlem}.
\begin{lemma}\label{tauplusrefl}
The functor $\tau_+$ reflects isomorphisms.
\end{lemma}
\begin{proof}
This is a consequence of Poincar{\'e} duality and Whitehead's Theorem.
\end{proof}
\begin{remark*}
For $n \geq 3$, let $[\frac{n}{2}]$ be the largest integer $\leq n$. Associating with a ${\rm PD}^n$--complex, $X$, the \emph{pre-fundamental triple} $(P_{[\frac{n}{2}]}X, \omega_X, p_{[\frac{n}{2}]\ast}[X])$, there is an analogue of  Lemma \ref{tauplusrefl}, namely, an orientation preserving map between ${\rm PD}^n$--complexes is a homotopy equivalence if and only if the induced map between pre-fundamental triples is an isomorphism. However, pre-fundamental triples do not determine the homotopy type of a ${\rm PD}^n$--complex as in Corollary \ref{class}, which is demonstrated by the fake products $X = (S^n \vee S^n) \cup_{\alpha} e^{2n}$, where $\alpha$ is the sum of the Whitehead product $[\iota_1 , \iota_2]$ and an element $\iota_1 \beta$ with $\beta \in \pi_{2n-1}(S^n)$ having trivial Hopf invariant. Pre-fundamental triples coincide with the fundamental triple for $n = 3$ and $n = 4$. It remains an open problem to enrich the structure of a pre-fundamental triple to obtain an analogue of Corollary \ref{class}.
\end{remark*}
\begin{lemma}\label{mainlem}
Let $X$ and $Y$ be standard ${\rm PD}^n$--complexes for $n \geq 4$ and weakly standard for $n = 3$ and let $f: \tau_+ Y \rightarrow \tau_+ X$ be a morphism in ${\rm\bf Trp}^n_+$.
Then $f$ is $\tau$--realizable by a map $\overline f: Y \rightarrow X$ in ${\rm PD}^n_+$ with $\tau \overline f = f$.
\end{lemma}
For the proof of Lemma \ref{mainlem}, we use 
\begin{lemma}\label{abargenerates}
Let $X = X' \cup e^n$ be a weakly standard ${\rm PD}^n$--complex. Then there is a generator $[e] \in \widehat C_n(X)$, with $\widehat C_nX = \widehat C_nX' \oplus \Lambda[e]$, corresponding to the cell $e^n$, such that $1 \otimes [e] \in \mathbb Z^{\omega} \otimes_{\Lambda} \widehat C_n X$ is a cycle representing the fundamental class $[X]$. Let $\{ e_m \}_{m \in M}$ be a basis of $\widehat C_{n-1} X = \widehat C_{n-1} X'$. Then the coefficients $\{ a_m \}_{m \in M}, a_m \in \Lambda$ for $m \in M$, of the linear combination $d_n[e] = \sum a_m [e_m]$   generate $\overline{I(\pi_1X)}$ as a right $\Lambda$--module.
\end{lemma}
\begin{proof}
Poincar{\'e} duality implies ${\rm H}_n(X, \mathbb Z^{\omega}) \cong H^0(X, \mathbb Z) \cong \mathbb Z$. Hence $1 \otimes d$ maps a multiple of the generator $1 \otimes [e]$ of $\mathbb Z^{\omega} \otimes_{\Lambda} \widehat C_n(X) = \mathbb Z^{\omega} \otimes_{\Lambda} \Lambda [e] \cong \mathbb Z$ to zero, that is, there is an $n \in \mathbb N$ such that
\begin{eqnarray*}
0
& = & 1 \otimes d(n(1 \otimes [e])) = n (1 \otimes d[e])
   = n (1 \otimes \sum_{m \in M} a_m [e_m]) \\
& = & n \sum 1.a_{m} \otimes [e_m] = n \sum_{m \in M} {\rm aug}(\overline{a_m}) \otimes [e_m].
\end{eqnarray*}
Since $\mathbb Z^{\omega} \otimes_{\Lambda}D_{n-1} = \mathbb Z^{\omega} \otimes_{\Lambda}\bigoplus_{m \in M} \Lambda [e_{m}] \cong \bigoplus_{m \in M} \mathbb Z^{\omega} \otimes_{\Lambda}\Lambda [e_{m}] = \bigoplus_{m \in M} \mathbb Z$ is free as abelian group, this implies ${\rm aug}(\overline{a_{m}}) = 0$ and hence $\overline{a_m} \in I$ for every $m \in M$. Therefore
$1 \otimes d (1 \otimes [e]) = 0$ and $1 \otimes [e] \in \mathbb Z^{\omega} \otimes_{\Lambda}D_n$
is a cycle representing a generator of the group ${\rm H}_n(X, \mathbb Z^{\omega})$. Without loss of generality we may assume that $e$ is oriented such that $1 \otimes e$ represents the fundamental class $[X]$. Further, Poincar{\'e} duality implies ${\rm H}^n(X, ^{\omega}\!\! \Lambda) \cong \mathbb Z$ and hence $I(\pi) \cong {\rm im}(d^{\ast})[e]$. But, for every $\varphi \in ^{\omega}\!{\rm Hom}_{\Lambda}(C_{n-1}, ^{\omega}\! \Lambda)$,
\begin{equation*}
\big(d^{\ast} \varphi \big) [e] = \varphi (d[e]) = \varphi \big(\sum a_m [e_m] \big) = \sum a_m \varphi [e_m] = \big(\sum\overline{\varphi [e_m]} \overline a_m [e]^{\ast}\big)[e],
\end{equation*}
where $[e]^{\ast}: {{\Lambda}}[e] \rightarrow {{\Lambda}}, [e] \mapsto 1$. Thus $I(\pi)$ is generated by $\{ \overline a_m \}_{m \in M}$ as a left $\Lambda$--module and hence $\overline{I(\pi)}$ is generated by $\{ a_m \}_{m \in M}$ as a right $\Lambda$--module
\end{proof}
\begin{lemma}\label{split}
Let $\overline X = X' \cup_f e^3$ be a weakly standard ${\rm PD}^3$--complex. Then we can choose a homotopy $f \simeq g$ such that $X = X' \cup_g e^3$ admits a splitting $\widehat C_2X = S \oplus d_3(\widehat C_3X')$ as a direct sum of $\Lambda$--modules satisfying $d_3[e] \in S$.
\end{lemma}
\begin{proof}
As $X'$ is homologically $2$--dimensional, $\widehat C(\overline X)$ admits a splitting,
\begin{equation*}
\widehat C_2(\overline X) = {\rm im} d'_3 \oplus S,
\end{equation*}
as a direct sum of $\Lambda$--modules, where $d_3': \widehat C_3(X') \rightarrow \widehat C_2(X')$. Thus $d_3[e] \in \widehat C_2(\overline X) = {\rm im} d'_3 \oplus S$ decomposes as a sum $d_3[e] = \alpha + \beta$ with $\alpha \in {\rm im} d'_3$ and $\beta \in S$. Since $\alpha$, viewed as a map $S^2 \rightarrow X'$, is homotopically trivial in $X'$, there is a homotopy $f \simeq g$, where $g$ represents $\beta$, such that $X = X' \cup_g e^3$ has the stated properties. 
\end{proof}

\begin{proof}[Proof of Lemma \ref{mainlem}]
Certain aspects of the proof for the case $n = 3$ differ from that for the case $n \geq 4$. Those parts of the proof pertaining to the case $n = 3$ appear in square brackets [ \ldots ]. [For $n = 3$ we assume that $X = X' \cup_g e^3$ is chosen as in Lemma \ref{split}.]

Given $X = X' \cup_g e^n$ and $Y = Y' \cup_{g'} {e'}^n$ and a morphism $\varphi = \{f\}: \tau(Y) = (P, \omega_Y, t_Y) \rightarrow \tau(X) = (Q, \omega_X, t_X)$ in ${\rm\bf Trp}^n$, the diagram
\begin{equation*}
\xymatrix{
\ \  \ \ \ \ X^{n-1} \subseteq X' \subset X \ar[r]^-p & P = P_{n-2}X \\
\ \ \ \ \ \ \ Y^{n-1} \subseteq Y' \ar[u]^{\overline \eta} \subset Y \ar[r]^-{p'} & Q = P_{n-2}Y, \ar[u]_f}
\end{equation*}
commutes in ${\rm\bf CW}_0$, where $p$ and $p'$ coincide with the identity morphisms on the $(n-1)$--skeleta, and where $\overline \eta$ is the restriction of $f$. For $n \geq 4$, we have $X' = X^{n-1}$ and $Y' = Y^{n-1}$. We obtain the following commutative diagram of chain complexes in ${\rm\bf H}_0$
\begin{equation*}
\xymatrix{
\ \ \ \ \ \ \ \widehat CX^{n-1} \subset \widehat CX \ar[r]^-{p_{\ast}} & \widehat CP \\
\ \ \ \ \ \ \ \widehat CY^{n-1} \ar[u]^{\overline \eta_{\ast}} \subset \widehat CY \ar[r]^-{p'_{\ast}} & \widehat CQ. \ar[u]_{f_{\ast}}}
\end{equation*}
For $n \geq 4$, we construct a morphism $(\xi, \eta): r(Y) \rightarrow r(X)$ in the category ${\rm\bf H}_{n-1}^c$ of homotopy systems of order $(n-1)$ (see Section \ref{homsys}), rendering the diagram
\begin{equation}\label{homcom}
\xymatrix{
r(X)  \ar[r]^-{r(p)} & r(P) \\
r(Y) \ar[r]^-{r(p')} \ar[u]^-{(\xi, \eta)} & r(Q) \ar[u]_-{r(f)}}
\end{equation}
homotopy commutative in ${\rm\bf H}_{n-1}^c$. Here $\xi: \widehat CY \rightarrow \widehat CX$ and $\eta: Y^{n-2} \rightarrow X^{n-2}$ is the restriction of $\overline \eta$ above. 

[For $n = 3$, the map $\overline \eta$ itself need not extend to a map $Y' \rightarrow X'$. But, since $Y'$ is homologically $2$--dimensional, there is a map $\eta': Y' \rightarrow X'$ inducing $\pi_1\eta' = \pi_1\varphi$. Since we may assume that $Q$ is obtained from $Y$ by attaching cells of dimension $\geq 3$, we can choose $f$ representing $\varphi$ with $p\eta' = fp'$.]

We write $\pi = \pi_1X, \pi' = \pi_1Y,  \Lambda = \mathbb Z[\pi]$ and $\Lambda' = \mathbb Z[\pi']$ and let $[e'] \in \widehat C_nY$ and $[e] \in \widehat C_nX$ be the elements corresponding to the $n$--cells $e_n$ and $e'_n$, respectively, $n \geq 3$. Since $\{ f \}$ is a morphism in ${\rm\bf Trp}^n$, we obtain $f_{\ast} p'_{\ast} [Y] = p_{\ast}[X]$ in ${\rm H}_n(P, \mathbb Z^{\omega})$ and hence
\begin{equation*}
f_{\ast}p'_{\ast} [e'] - p_{\ast} [e] \in {\rm im}(d: \widehat C_{n+1}P \rightarrow \widehat C_nP) + \overline{I(\pi)}\widehat C_nP.
\end{equation*}
Thus there are elements $x \in \widehat C_{n+1}P$ and $y \in \overline{I(\pi)}\widehat C_nP$ with
\begin{equation}\label{xydfn}
f_{\ast}p'_{\ast} [e'] - p_{\ast} [e] = dx + y.
\end{equation}
Let $\{ e'_m\}_{m \in M}$ be a basis of $\widehat C_{n-1}Y$. By Lemma \ref{abargenerates}, 
\begin{equation}\label{de'}
d [e'] = \sum a_m [e'_m],  
\end{equation}
for some $a_m \in \Lambda', m \in M$, where $\{ a_m \}_{m \in M}$ generate $\overline{I(\pi')}$ as right $\Lambda'$--module. Since $\varphi = f_{\ast}$ is surjective, $\overline{I(\pi)}$ is generated by $\{ \varphi(a_m) \}_{m \in M}$ as right $\Lambda$--module, and we may write
\begin{equation}\label{zmdfn}
y = \sum_{m \in M} \varphi(a_m) z_m,
\end{equation}
for some $z_m \in \widehat C_nP, m \in M$, since there is a surjection $\bigoplus_{m \in M} \Lambda [m] \twoheadrightarrow \overline{I(\pi)}$ of right $\Lambda$--modules which maps the generator $[m]$ to $\varphi(a_m)$. Then (\ref{xydfn}) implies $d( f_{\ast}p'_{\ast} [e'] - p_{\ast} [e])  = d y = \sum_{m \in M} \varphi(a_m) d z_m$, and hence
\begin{equation}\label{pastde}
p_{\ast} d [e] = \sum_{m \in M} \varphi(a_m) f_{\ast} p'_{\ast} [e_m'] - \sum_{m \in M} \varphi(a_m) d z_m.
\end{equation}
We define the $\varphi$--equivariant homomorphism
\begin{equation}\label{overlinealpha}
\overline \alpha_n: \widehat C_{n-1}Y \rightarrow \widehat C_nP \quad {\rm by} \quad
\overline\alpha_n([e_m']) = -z_m.
\end{equation}
For $n \geq 4$, we define $\xi: \widehat CY \rightarrow \widehat CX$ by $\xi[e'] = [e]$ and 
\begin{equation}\label{xidfn}
\xi_i = 
\begin{cases}
\widehat C_{n-1}(\overline \eta) + d \overline \alpha_n & {\rm for} \ i = n-1, \\
\widehat C_i(\eta) & {\rm for} \ i < n-1. \\
\end{cases}
\end{equation} 
[For $n = 3$ we use the splitting $\widehat C_2Y = S \oplus d_3 \widehat C_3Y'$ in Lemma \ref{split} and define $\xi_i: \widehat C_iY \rightarrow \widehat C_iX$ by $\xi_3[e'] = [e], \xi_3 | \widehat C_3Y' = \widehat C_3 \eta'$, and
\begin{eqnarray*}
&& \xi_2 | S = (\widehat C_2 \eta' + d \overline \alpha_3) | S, \\
&& \xi_2 | d_3 \widehat C_3Y' = \widehat C_2\eta' | d_3 \widehat C_3Y', \\
&&\xi_i = \widehat C_i \eta \quad {\rm for} \ i <2.]
\end{eqnarray*}

To ensure that $\xi$ is a chain map, it is now enough to show that $d\xi[e'] = \xi d[e']$. But, for the injection $\widehat C(p) = p_{\ast}$, we obtain
\begin{eqnarray*}
p_{\ast} \xi d [e']
& = &  p_{\ast} ( \widehat C_{n-1} (\eta') + d \overline \alpha_n) d[e']\\
& = & (p \circ \eta')_{n-1}d[e']+ p_{\ast}\big(d \overline \alpha_n (\sum_{m \in M} a_m [e'_m])\big) \\
& = & (f \circ p')_{n-1} d[e']  + p_{\ast} \sum_{m \in M} \varphi(a_m) d \overline \alpha_n [e'_m] \\
& = & \sum_{m \in M} \varphi(a_m) f_{\ast}p'_{\ast} [e'_m] - p_{\ast} \sum_{m \in M} \varphi(a_m) d z_m \\
& = & p_{\ast} d [e] = p_{\ast} d \xi [e'], \quad {\rm by} \ (\ref{pastde}). 
\end{eqnarray*}

[For $n = 3$, Theorem \ref{pd3funcrefisofull} now implies that there is a map $\overline f: Y \rightarrow X$ such that $\widehat C(\overline f) = \xi$. Then $\tau(\overline f) = f$, $\overline f$ is a degree $1$ map and the proof is complete for $n = 3$.] 

Now let $n \geq 4$. To check that $(\xi, \eta)$ is a morphism in ${\rm \bf H}_{n-1}^c$, note that the attaching map satisfies the cocycle condition and hence, by its definition, the map $\xi_{n-1}$ commutes with attaching maps in $r(X)$ and $r(Y)$, since $\widehat C_{n-1}\overline \eta$ has this property. We must show that Diagram (\ref{homcom}) is homotopy commutative. But $r(f) = (f_{\ast}, \eta)$ and $r(p) = (p_{\ast}, j), r(p') = (p'_{\ast}, j')$, where $j$ and $j'$ are the identity morphisms on $X^{n-2} = P^{n-2}$ and $Y^{n-2} = Q^{n-2}$, respectively. Hence we must find a homotopy $\alpha: (p_{\ast} \xi, \eta) \simeq (f_{\ast} p'_{\ast}, \eta)$ in ${\rm\bf H}_{n-1}^c$, that is, $\varphi$--equivariant maps 
\begin{equation*}
\alpha_{i+1}: \widehat C_iY \rightarrow \widehat C_{i+1}P, \ i \geq n-1,
\end{equation*} 
such that 
\begin{eqnarray}
&& \{ \eta \} + g_{n-1}\alpha_{n-1} = \{ \eta \}, \label{alphafirstcond}\\
&& (p_{\ast} \xi)_i - (f \circ p')_i = \alpha_i d + d \alpha_{i+1} \quad {\rm for} \quad i \geq n-1, \label{alphaseccond}
\end{eqnarray}
where $g_{n-1}$ is the attaching map of $(n-1)$--cells in $P$. Define $\alpha$ by $\alpha_{n+1}[e'] = - x$, see (\ref{xydfn}), and
\begin{equation}\label{alphadfn}
\alpha_i =
\begin{cases}
\overline \alpha_n & {\rm for} \ i= n, \\
0  & {\rm for} \ i < n. 
\end{cases}
\end{equation}
Then $\alpha$ satisfies (\ref{alphafirstcond}) trivially. For $i = n -1$, we obtain
\begin{eqnarray*}
(p_{\ast} \xi)_{n-1} - (f \circ p')_{n-1} 
& = & \xi_{n-1} - \widehat C_{n-1}(f) \\
& = & \xi_{n-1} - \widehat C_{n-1}(\overline \eta) \\
& =  & d \alpha_n, \quad {\rm by \ (\ref{xidfn}) \ and \ (\ref{alphadfn}).}
\end{eqnarray*}
For $i = n$, we evaluate (\ref{alphaseccond}) on $[e']$. By (\ref{xydfn}),
\begin{equation*}
(p_{\ast} \xi - f_{\ast}p'_{\ast})[e'] 
= p_{\ast} [e] - f_{\ast}p'_{\ast}[e'] =  - dx - y.
\end{equation*}
On the other hand,
\begin{eqnarray*}
(d \alpha_{n+1} + \alpha_n d)[e'] 
& = & d \alpha_{n+1}[e'] + \alpha_n \sum_{m \in M} a_m [e'_m], \quad {\rm by} \ (\ref{de'}), \\
& = & - dx - \sum_{m \in M} \varphi(a_m) z_m, \quad {\rm by} \ (\ref{alphadfn}) \ {\rm and} \ (\ref{overlinealpha}), \\
& = & - dx - y \quad {\rm by} \ (\ref{zmdfn}).
\end{eqnarray*}
Hence $\alpha$ satisfies (\ref{alphaseccond}) and Diagram (\ref{homcom}) is homotopy commutative.

To construct a morphism $\overline f: Y \rightarrow X$ in ${\rm\bf PD}^n_+$ with $\tau(\overline f) = f$, consider the obstruction $\mathcal O(\xi, \eta) \in {\rm H}^n(Y, \Gamma_{n-1}X)$ (see Section \ref{homsys}) and note that $p$ induces an isomorphism $p_{\ast}: \Gamma_{n-1}X \rightarrow \Gamma_{n-1}P$, see II.4.8 \cite{Baues1}. Hence the obstruction for the composite $r(p) (\xi, \eta)$ coincides with $p_{\ast} \mathcal O(\xi, \eta)$, where $p_{\ast}$ is an isomorphism. On the other hand, the obstruction for $r(f) r(p')$ vanishes, since this map is $\lambda$--realizable. Thus, by the homotopy commutativity of (\ref{homcom}), $p_{\ast} \mathcal O(\xi, \eta) = \mathcal O (r(f) r(p')) = 0$, so that $\mathcal O(\xi, \eta) = 0$ and there is a $\lambda$--realization $(\xi, \tilde \eta')$ of $(\xi, \eta)$ in ${\rm\bf H}_n^c$. Since ${\rm H}^{n+1}(Y, \Gamma_nX) = 0$, there is a $\lambda$--realization $(\xi, \overline f)$ of $(\xi, \tilde \eta')$ in ${\rm\bf H}_{n+1}^c$. As $Y = Y^n, X = X^n$ and $\xi$ is compatible with fundamental classes by construction, $\overline f: Y \rightarrow X$ is a degree $1$ map in ${\rm\bf PD}^n_+$ realizing the map $f$ in ${\rm \bf Trp}^n$.

\end{proof}

\section{$\rm PD^3$--complexes}\label{pd3section}
The fundamental triple of a ${\rm PD}^3$--complex consists of a group $\pi$, an orientation $\omega$ and an element $t \in {\rm H}_3(\pi, \mathbb Z^{\omega})$. Here we use the standard fact that the homology of a group $\pi$ coincides with the homology of the corresponding Eilenberg--{Mac\,Lane} space $K(\pi,1)$. In general, it is a difficult problem to actually compute ${\rm H}_3(\pi, \mathbb Z^{\omega})$. The homotopy type of a ${\rm PD}^3$--complex is characterized by its fundamental triple, but not every fundamental triple occurs as the fundamental triple of a ${\rm PD}^3$--complex. Via the invariant $\nu_C(t)$ Turaev \cite{Turaev} characterizes those fundamental triples which are realizable by a ${\rm PD}^3$--complex. Let ${\rm\bf Trp}^3_{+,\nu}$ be the full subcategory of ${\rm\bf Trp}^3_+$ consisting of fundamental triples satisfying Turaev's realization condition. Then Theorem \ref{mainthm} implies

\begin{theorem}\label{tauplus3}
The functor
\begin{equation*}
\tau_+: {\rm\bf PD}^3_+ / \simeq \rightarrow {\rm\bf Trp}^3_{+, \nu}
\end{equation*}
reflects isomorphisms and is representative and full.
\end{theorem}

\begin{remark*}
Turaev does not mention that the functor $\tau_+$ is actually full and thus only proves the first part of the following corollary which is one of the consequences of Theorem \ref{tauplus3}.
\end{remark*}

\begin{corollary}
The functor $\tau_+$ yields a 1--1 correspondence between oriented homotopy types of ${\rm PD}^3$--complexes and isomorphism types of fundamental triples satisfying Turaev's realization condition. Moreover, for every ${\rm PD}^3$--complex $X$, there is a surjection of groups
\begin{equation*}
\tau_+: {\rm Aut}_+(X) \rightarrow {\rm Aut}(\tau(X)),
\end{equation*}
where ${\rm Aut}_+(X)$ is the group of oriented homotopy equivalences of $X$ in ${\rm\bf PD}^3_+ / \simeq$ and ${\rm Aut}(\tau(X))$ is the group of automorphisms of the triple $\tau(X)$ in ${\rm\bf Trp}^3_+$ which is a subgroup of ${\rm Aut}(\pi_1X)$.
\end{corollary}

As every $3$--manifold has the homotopy type of a finite standard ${\rm PD}^3$--complex, the question arises which fundamental triples in ${\rm\bf Trp}^3_+$ correspond to finite standard ${\rm PD}^3$--complexes. While Turaev does not discuss this question, we use the concept of ${\rm PD}^3$--chain complexes (see Section \ref{PDcomp}) in the category ${\rm\bf PD}_{\ast}^3$ to do so.

\begin{theorem}\label{pd3funcrefisofull}
The functor $\widehat C: {\rm\bf PD}^3/\simeq \ \longrightarrow {\rm\bf PD}^3_{\ast}/\simeq$ reflects isomorphisms and is representative and full.
\end{theorem}

\begin{proof}
This follows from Theorems \ref{pdnfuncequ} and \ref{pdnfuncdet} in Section \ref{pdhomsys}.
\end{proof}

\begin{corollary}\label{onetoonepd3}
The functor $\widehat C$ yields a 1--1 correspondence between homotopy types of ${\rm PD}^3$--complexes and homotopy types of ${\rm PD}^3$--chain complexes. Moreover, for every ${\rm PD}^3$--complex $X$ there is a surjection of groups
\begin{equation*}
\widehat C: {\rm Aut}(X) \longrightarrow {\rm Aut}(\widehat C(X)).
\end{equation*}
\end{corollary}

\begin{remark}
Corollary \ref{onetoonepd3} implies that the diagonal of every ${\rm PD}^3$--chain complex is, in fact, homotopy co--associative and homotopy co--commutative.
\end{remark}

Connecting the functor $\widehat C$ and the functor $\tau_+$, we obtain the diagram
\begin{equation*}
\xymatrix{
{\rm\bf PD}^3_+ / \simeq \ar[rr]^{\widehat C} \ar[dr]_{\tau_+} && 
{\rm\bf PD}^3_{\ast+} / \simeq \ar[dl]^{\tau_{\ast}} \\
& {\rm\bf Trp}^3_{+, \nu} &}
\end{equation*}
where $\tau_+$ determines $\tau_{\ast}$ together with a natural isomorphism $\tau_{\ast} \widehat C \cong \tau_+$.

\begin{corollary}
All of the functors $\widehat C, \tau_+$ and $\tau_{\ast}$ reflect isomorphisms and are full and representative.
\end{corollary}

By Remark \ref{corr}, the functor $\widehat C$ yields a  1--1 correspondence between homotopy types of finite standard ${\rm PD}^3$--complexes and finite standard ${\rm PD}^3$--chain complexes, respectively.

\section{Realizability of $\rm PD^4$--chain complexes}\label{4chainrealize}
Given a ${\rm PD}^4$--chain complex $C$, we define an invariant $\mathcal O (C)$ which vanishes if and only if $C$ is realizable by a ${\rm PD}^4$--complex. To this end we recall the \emph{quadratic functor $\Gamma$} (see also (4.1) p. 13 in \cite{Baues1}). A function $f: A \rightarrow B$ between abelian groups is called a \emph{quadratic map} if $f(-a) = f(a)$, for $a \in A$, and if the function $A \times A \rightarrow B, (a,b) \mapsto f(a+b) - f(a) - f(b)$ is bilinear. There is a \emph{universal quadratic map}
\begin{equation*}
\gamma: A \rightarrow \Gamma(A),
\end{equation*}
such that for all quadratic maps $f: A \rightarrow B$ there is a unique homomorphism $f^{\square}: \Gamma(A) \rightarrow B$ satisfying $f^{\square} \gamma = f$. Using the cross effect of $\gamma$, we obtain the \emph{Whitehead product map} 
\begin{eqnarray*}
P: A \otimes A & \longrightarrow & \Gamma(A), \\
a \otimes b & \longmapsto & [a,b] = \gamma(a+b) - \gamma(a) - \gamma(b).
\end{eqnarray*}
The \emph{exterior product $\Lambda^2A$} of the abelian group $A$ is defined so that we obtain the natural exact sequence
\begin{equation}\label{themapH}
\Gamma(A) \stackrel{H}{\longrightarrow} A \otimes A \longrightarrow \Lambda^2 A \longrightarrow 0,
\end{equation}
where $H$ maps $\gamma(a)$ to $a \otimes a$ for $a \in A$ (see also p.14 in \cite{Baues1}). The composite $P H: \Gamma(A) \rightarrow \Gamma(A)$ coincides with $2{\rm id}_{\Gamma(A)}$, in fact, $P H$ maps $\gamma(a)$ to $[a,a] = 2\gamma(a)$. Given a ${\rm CW}$--complex $X$, there is a natural isomorphism $\Gamma_3(X) \cong \Gamma(\pi_2X)$, by an old result of J.H.C. Whitehead \cite{Whitehead2}, where $\Gamma_3$ is \emph{Whitehead's functor} in \emph{A Certain Exact Sequence} \cite{Whitehead2}. 

\begin{theorem}\label{pd4real}
Let $C = ( (\pi, C), \omega, [C], \Delta)$ be a ${\rm PD}^4$--chain complex with homology module ${\rm H}_2(C, \Lambda) = H_2$. Then there is an invariant
\begin{equation*}
\mathcal O(C) \in {\rm H}_0(\pi, \Lambda^2H_2^{\omega})
\end{equation*}
with $\mathcal O(C) = 0$ if and only if there is a ${\rm PD}^4$--complex $X$ such that $\widehat C(X)$ is isomorphic to $C$ in ${\rm\bf PD}^4_{\ast} / \simeq$. Moreover, if $\mathcal O(C) = 0$, the group 
\begin{equation*}
\ker \big(H_{\ast}: {\rm H}_0(\pi, \Gamma(H_2^{\omega}) \longrightarrow {\rm H}_0(\pi, H_2^{\omega} \otimes H_2^{\omega}) \big)
\end{equation*}
acts transitively and effectively on the set ${\rm Real}_{\widehat C}(C)$ of realizations of $C$ in ${\rm\bf PD}^4 / \simeq$. Here $\ker H_{\ast}$ is $2$--torsion.
\end{theorem}

\begin{proof}
First note that
\begin{equation}\label{isos}
{\rm H}^4(C, \Lambda^2 H_2) \cong {\rm H}_0(C, \Lambda^2 H_2^{\omega}) \cong {\rm H}_0(\pi, \Lambda^2 H_2^{\omega}).
\end{equation}
By Lemma \ref{2real}, we may assume that $C$ is $2$--realizable. By Proposition \ref{objrealize}, there is thus a $4$--dimensional ${\rm CW}$--complex $X$ together with an isomorphism $\widehat CX \cong (\pi, C)$. The ${\rm CW}$--complex $X$ yields the homotopy systems $\overline{\overline X}$ in ${\rm\bf H}_3^c$ and $\overline X$ in ${\rm\bf H}_4^c$ with $\overline X = r(X)$ and $\overline{\overline X} = \lambda X$. By Theorem \ref{pdnfuncequ}, we may choose a diagonal $\overline{\overline \Delta}: \overline{\overline X} \rightarrow \overline{\overline X} \otimes \overline{\overline X}$ inducing $\Delta: C \rightarrow C \otimes C$, whose homotopy class is determined by $\Delta$. However, $\overline{\overline \Delta}$ need not be $\lambda$--realizable. Lemma \ref{diagreal0} shows that there is an obstruction
\begin{equation}
\mathcal O' = \mathcal O_{\overline X, \overline X \otimes \overline X} (\overline{\overline \Delta})\in {\rm H}^4(C, \Gamma_3(\overline X \otimes \overline X))
\end{equation}
which vanishes if and only if there is a diagonal $\overline \Delta: \overline X \rightarrow \overline X \otimes \overline X$ realizing $\overline{\overline \Delta}$. Note that $\mathcal O'$ is determined by the diagonal $\Delta$ on $C$, since the obstruction depends on the homotopy class of $\overline{\overline \Delta}$ only. By Theorem \ref{pdnfuncdet}, the existence of  $\overline \Delta$ realizing $\overline{\overline \Delta}$ also implies the existence of $\Delta: X \rightarrow X \times X$ realizing $\overline \Delta$. But
\begin{eqnarray*}
\Gamma_3(\overline X \otimes \overline X) 
& \cong & \Gamma(\pi_2(\overline X \otimes \overline X)) \\
& \cong & \Gamma(\pi_2(X \times X)) \\
& \cong & \Gamma(\pi_2 \oplus \pi_2) \quad {\rm where} \ \pi_2 = \pi_2X.
\end{eqnarray*}
Applying Lemma \ref{diagreal1} (1), we see that
\begin{equation*}
\mathcal O' \in \ker p_{i\ast} \quad {\rm for} \ i = 1, 2,
\end{equation*}
where $p_i: \pi_2 \oplus \pi_2 \rightarrow \pi_2$ is the $i$--th projection. Now 
\begin{equation*}
\Gamma(\pi_2 \oplus \pi_2) = \Gamma(\pi_2) \oplus \pi_2 \otimes \pi_2 \oplus \Gamma(\pi_2)
\end{equation*}
and hence $\mathcal O'$ yields $\mathcal O'' \in {\rm H}^4(C, \pi_2 \otimes \pi_2)$. While the homotopy type of $\overline{\overline X}$ is determined by $C$, the homtopy type of $\overline X$ is an element of ${\rm Real}_{\lambda}(\overline {\overline X})$ and the group ${\rm H}^4(C, \Gamma(\pi_2))$ acts transitively and effectively on this set of realizations. To describe the behaviour of the obstruction under this action using Lemma \ref{obsact}, we first consider the homomorphism
\begin{equation*}
\nabla = \Delta_{\ast} - \iota_{1\ast} - \iota_{2\ast}: \Gamma(\pi_2) \longrightarrow \Gamma(\pi_2 \oplus \pi_2),
\end{equation*}
where $\Delta: \pi_2 \rightarrow \pi_2 \oplus \pi_2$ maps $x \in \pi_2$ to $\iota_1(x) + \iota_2(x)$. We obtain, for $x \in \pi_2$,
\begin{eqnarray*}
\nabla(\gamma(x))
& = & \gamma(\iota_1(x) + \iota_2(x)) - \gamma(\iota_1(x)) - \gamma(\iota_2(x)) \\
& = & [ \iota_1(x) , \iota_2(x) ] \\
& = & x \otimes x \in \pi_2 \otimes \pi_2 \subset \Gamma(\pi_2 \oplus \pi_2),
\end{eqnarray*}
showing that $\nabla$ coincides with $H: \Gamma(\pi_2) \rightarrow \pi_2 \otimes \pi_2.$ Given $\alpha \in {\rm H}^4(C, \Gamma(\pi_2))$, the obstruction $\mathcal O''_{\alpha} = \mathcal O_{\overline Y, \overline Y \otimes \overline Y}(\overline{\overline \Delta})$ with $\overline Y = \overline X + \alpha$ satisfies
\begin{equation*}
\mathcal O''_{\alpha} = \mathcal O'' + H_{\ast}\alpha,
\end{equation*}
by Lemma \ref{obsact}. The exact sequence
\begin{equation*}
{\rm H}^4(C, \Gamma(\pi_2)) \longrightarrow {\rm H}^4(C, \pi_2 \otimes \pi_2) \longrightarrow {\rm H}^4(C, \Lambda^2\pi_2) \longrightarrow 0
\end{equation*}
allows us to identify the coset of ${\rm im} H_{\ast}$ represented by $\mathcal O''$ with an element
\begin{equation*}
\mathcal O \in {\rm H}^4(C, \Lambda^2H_2),
\end{equation*}
where $H_2 = {\rm H}_2(C, \Lambda) \cong \pi_2$. By the isomorphisms (\ref{isos}), this element yields the invariant 
\begin{equation*}
\mathcal O \in {\rm H}_0(\pi, \Lambda^2H_2^{\omega})
\end{equation*}
with the properties stated. Given that $\mathcal O''$ vanishes, the obstruction $\mathcal O_{\alpha}''$ vanishes if and only if $\alpha \in \ker H_{\ast}$, and Proposition \ref{objrealize} yields the result on ${\rm Real}_{\widehat C}(C)$. We observe that $\ker H_{\ast}$ is $2$--torsion as $H_{\ast}(x) = 0$ implies $2x = P_{\ast}H_{\ast} x = 0$.
\end{proof}

\begin{theorem}
Let $C = ((\pi, C), \omega, [C], \Delta)$ be a ${\rm PD}^4$--chain complex for which $\Delta$ is homotopy co-commutative. Then the obstruction $\mathcal O(C)$ is $2$--torsion, that is, $2 \mathcal O(C) = 0$.
\end{theorem}

\begin{proof}
Lemma \ref{diagreal1} (2) states
\begin{equation*}
\mathcal O' \in \ker ({\rm id}_{\ast} - T_{\ast})_{\ast},
\end{equation*}
where ${\rm id}$ is the identity on $\pi_2 \oplus \pi_2$ and $T$ is the interchange map on $\pi_2 \oplus \pi_2$ with $T \iota_1 = \iota_2$ and $T \iota_2 = \iota_1$. Thus $T$ induces the map $-{\rm id}$ on $\Lambda^2 \pi_2$ and the result follows.
\end{proof}

\begin{remark*}
Lemma \ref{diagreal1} (3) concerning homotopy associativity of the diagonal does not yield a restriction of the invariant $\mathcal O(C)$.
\end{remark*}

\begin{theorem}\label{finodd}
The functor $\widehat C$ induces a 1--1 correspondence between homotopy types of ${\rm PD}^4$--complexes with finite fundamental group of odd order and homotoppy types of ${\rm PD}^4$--chain complexes with homotopy co--commutative diagonal and finite fundamental group of odd order.
\end{theorem}

\begin{proof}
Since $\pi$ is of odd order, the cohomology ${\rm H}^0(\pi, M)$ is odd torsion and the result follows from Theorem \ref{pd4real}.
\end{proof}

\begin{remark*}
By Theorem \ref{finodd}, every ${\rm PD}^4$--chain complex with homotopy co--com-mutative diagonal and odd fundamental group has a homotopy co--associative diagonal.
\end{remark*}

Up to $2$--torsion, Theorem \ref{pd4real} yields a correspondence between homotopy types of ${\rm PD}^4$--complexes and homotopy types of ${\rm PD}^4$--chain complexes. In Section \ref{algmodPD4} below we provide a precise condition for a ${\rm PD}^4$--chain complex to be realizable by a ${\rm PD}^4$--complex.

\section{The chains of a $2$--type}\label{2typechains}
The fundamental triple of a ${\rm PD}^4$--complex $X$ comprises its $2$--type $T = P_2X$ and an element of the homology ${\rm H}_4(T, \mathbb Z^{\omega})$. To compute ${\rm H}_4(T, \mathbb Z^{\omega})$, we construct a chain complex $P(T)$ which approximates the chain complex $\widehat C(T)$ up to dimension $4$. Our construction uses a presentation of the fundamental group as well as the concepts of \emph{pre--crossed module} and \emph{Peiffer commutator}. To introduce these concepts, we work with right group actions as in \cite{Baues1}, and define $P(T)$ as a chain complex of right $\Lambda$--modules. With any left $\Lambda$--module $M$ we associate a right $\Lambda$--module in the usual way by setting $x.\alpha = \alpha^{-1}.x $, for $\alpha \in \pi$ and $x \in M$, and vice versa.

A \emph{pre--crossed module} is a group homomorphism $\partial: \rho_2 \rightarrow \rho_1$ together with a right action of $\rho_1$ on $\rho_2$, such that
\begin{equation*}
\partial(x^{\alpha}) = - \alpha + \partial x + \alpha \quad {\rm for} \ x \in \rho_2, \alpha \in \rho_1,
\end{equation*}
where we use additive notation for the group law in $\rho_1$ and $\rho_2$, as in \cite{Baues1}. For $x, y \in \rho_2$, the \emph{Peiffer commutator} is given by
\begin{equation*}
\langle x, y \rangle = -x - y + x + y^{\partial x}.
\end{equation*}
A pre--crossed module is a \emph{crossed module}, if all Peiffer commutators vanish. A \emph{map of pre--crossed modules}, $(m,n): \partial \rightarrow \partial'$ is given by a commutative diagram
\begin{equation*}
\xymatrix{
\rho_2 \ar[rr]^m \ar[d]_{\partial} && \rho_2' \ar[d]^{\partial'} \\
\rho_1 \ar[rr]^n && \rho_1'}
\end{equation*}
in the category of groups, where $m$ is $n$--equivariant. Let ${\rm\bf cross}$ be the category of crossed modules and such morphisms. A \emph{weak equivalence} in ${\rm\bf cross}$ is a map $(m,n): \partial \rightarrow \partial'$, which induces isomorphisms ${\rm coker}\partial \cong {\rm coker}\partial'$ and $\ker\partial \cong \ker \partial'$, and we denote the localization of ${\rm\bf cross}$ with respect to weak equivalences by ${\rm\bf Ho(cross)}$. By an old result of Whitehead--{Mac\,Lane}, there is an equivalence of categories
\begin{equation*}
\overline \rho: 2-{\rm\bf{types}} \longrightarrow {\rm\bf Ho(cross)},
\end{equation*}
compare Theorem III 8.2 in \cite{Baues1}. The functor $\overline\rho$ carries a $2$--type $T$ to the crossed module $\partial: \pi_2(T, T^1) \rightarrow \pi_1(T^1).$

A pre--crossed module is \emph{totally free}, if $\rho_1 = \langle E_1 \rangle$ is a free group generated by a set $E_1$ and $\rho_2 = \langle E_2 \times \rho_1 \rangle$ is a free group generated by a free $\rho_1$--set $E_2 \times \rho_1$ with the obvious right action of $\rho_1$. A function $f: E_2 \rightarrow \langle E_1 \rangle$ yields the \emph{associated totally free pre--crossed module} $\partial_f: \rho_2 \rightarrow \rho_1$ with $\partial_f(x) = f(x)$ for $x \in E_2$. Let $Pei_n(\partial_f) \subset \rho_2$ be the subgroup generated by $n$--fold Peiffer commutators and put $\overline \rho_2 = \rho_2 / Pei_2(\partial_f)$. Let ${\rm\bf cross}^=$ be the category whose objects are pairs $(\partial_f, B)$, where $\partial_f$ is a totally free pre--crossed module $\partial_f: \rho_2 \rightarrow \rho_1$ and $B$ is a submodule of $\ker(\partial: \overline\rho_2 \rightarrow \rho_1$). Further, a morphism $m: (\partial_f, B) \rightarrow (\partial_{f'}, B')$ in ${\rm\bf cross}^=$ is a map $\partial_f \rightarrow \partial_{f'}$ which maps $B$ into $B'$. Then there is a functor 
\begin{equation*}
q: {\rm\bf cross}^= \longrightarrow {\rm\bf cross} \longrightarrow {\rm\bf Ho(cross)},
\end{equation*}
which assigns to $(\partial_f, B)$ the crossed module $\overline \rho_2 / B \rightarrow \rho_1$, and one can check that $q$ is full and representative. Given any map $g: T \rightarrow T'$ between $2$--types, we may choose a map $\overline{\overline g}: (\partial_f, B) \rightarrow (\partial_{f'}, B')$ in ${\rm\bf cross}^=$ representing the homotopy class of $g$ via the functor $q$ and the equivalence $\overline \rho$. We call $\overline{\overline g}$ a \emph{map associated with} $g$. 

Given an action of the group $\pi$ on the group $M$ and a group homomorphism $\varphi: N \rightarrow \pi$, a \emph{$\varphi$--crossed homomorphism} $h: N \rightarrow M$ is a function satisfying
\begin{equation*}
h(x + y) = (h(x))^{\varphi(y)} + h(y) \quad {\rm for} \ x, y \in N.
\end{equation*}
By an old result of Whitehead \cite{Whitehead1}, the totally free crossed module $\overline \rho_2 \rightarrow \rho_1$ enjoys the following properties. 

\begin{lemma}\label{whitehead}
Let $X^2$ be a $2$--dimensional ${\rm CW}$--complex in ${\rm\bf CW}_0$ with attaching map of $2$--cells $f: E_2 \rightarrow \langle E_1 \rangle = \pi_1(X^1)$. Then there is a commutative diagram
\begin{equation*} 
\xymatrix{
\pi_2(X^2, X^1) \ar[rr]^{\partial} \ar@{=}[d] && \pi_1(X^1) \ar@{=}[d] \\
\overline \rho_2 \ar[rr]^{\partial_f} && \rho_1,}
\end{equation*}
identifying $\partial$ with the totally free crossed module $\partial_f$. Moreover, the abelianization of $\overline \rho_2$ coincides with $\widehat C_2(X^2)$, identifying the kernel of $\partial_f$ with the kernel of $d_2: \widehat C_2(X^2) \rightarrow \widehat C_1(X^2)$, and $\partial_f$ determines the boundary $d_2$ via the commutative diagram
\begin{equation*}
\xymatrix{
\overline \rho_2 \ar[rr]^{\partial} \ar[d]_{h_2} && \rho_1 \ar[d]^{h_1} \\
\widehat C_2(X^2) \ar[rr]^d && \widehat C_1(X^2).}
\end{equation*}
Here $h_2$ is the quotient map and $h_1$ is the $(q: \rho_1 \rightarrow \pi_1(X^2))$--crossed homomorphism which is the identity on the generating set $E_1$. Each map $\partial_f \rightarrow \partial_{f'}$ induces a chain map $\widehat C_2(X^2) \rightarrow \widehat C_2(X'^2)$ where $X^2$ and $X'^2$ are the $2$--dimensional ${\rm CW}$--complexes with attaching maps $f$ and $f'$, respectively.
\end{lemma}

In addition to Lemma \ref{whitehead}, we need the following result on Peiffer commutators, which was originally proved in IV (1.8) of \cite{Baues1} and generalized in a paper with Conduch{\'e} \cite{BauesConduche}.

\begin{lemma}\label{omega}
With the notation in Lemma \ref{whitehead}, there is a short exact sequence
\begin{equation*}
0 \longrightarrow \Gamma(K) \longrightarrow \widehat C_2(X^2) \otimes \widehat C_2(X^2) \stackrel{\omega}{\longrightarrow} Pei_2(\partial_f) / Pei_3(\partial_f) \longrightarrow 0,
\end{equation*}
where $K = \ker d_2 = \pi_2 X^2$ and $\omega$ maps $x \otimes y$ to the Peiffer commutator $\langle \xi, \eta \rangle$ with $\xi, \eta \in \rho_2$ representing $x$ and $y$, respectively.
\end{lemma}

\begin{definition}
Given a $2$--type $T$ in $2$--{\bf{types}}, we define the chain complex $P(T) = P(\partial_f, B)$ as follows. Let $f: E_2 \rightarrow  \langle E_1 \rangle$ be the attaching map of $2$--cells in $T$ and put $C_i = \widehat C_i(T)$. Then the $2$--skeleton of $P(T)$ coincides with $\widehat C(T^2)$, that is, $P_i(T) = C_i$ for $i \leq 2$, and $P_i(T) = 0$ for $i > 4$. To define $P_4(T)$, let $H$ be the map in (\ref{themapH}) and put $B = {\rm im}(d: C_3 \rightarrow C_2)$ and $\nabla_B = B \otimes B + H[B, C_2]$ as a submodule of $C_2 \otimes C_2$. Then $P_4(T)$ is given by the quotient
\begin{equation*}
P_4(T) = C_2 \otimes C_2 / \nabla_B.
\end{equation*}
To define $P_3(T)$, we use Lemma \ref{whitehead}, Lemma \ref{omega} and the identification $\pi_2T^2 = \ker(d: C_2 \rightarrow C_1)$ and put $\sigma_2 = \rho_2 / Pei_3(\partial_f)$. Then $P_3(T)$ is given by the pull--back diagram
\begin{equation*}
\xymatrix{
P_3(T) \ \ar@{>->}[rr] \ar[d]_{\overline d} && \sigma_2 / \omega \nabla_B \ar@{->>}[d]\\
B \ \ar@{>->}[r] & \pi_2T^2 \ \ar@{>->}[r] & \overline \rho_2.}
\end{equation*}
The chain complex $P(T)$ is determined  by the commutative diagram
\begin{equation*}
\xymatrix{
P_4(T) \ar[r]^d \ar@{=}[d] & P_3(T) \ar@{->>}[dr]^{\overline d} \ar[rr] \ar@{>->}[d]&& 
P_2(T) \ar[r] \ar@{=}[d] & P_1(T) \ar@{=}[d] \ar[r] & P_0(T) \ar@{=}[d] \\
C_2 \otimes C_2 / \nabla_B \ar[r]^{-\omega} & \sigma_2 / \omega \nabla_B & B \ \ar@{>->}[r] & 
C_2 \ar[r] & C_1 \ar[r] & C_0.}
\end{equation*}
\end{definition}

Clearly, $P(T) = P(\partial_f, B)$ depends on the pair $(\partial_f, B)$ only and yields a functor
\begin{equation*}
P: {\rm\bf cross}^= \longrightarrow {\rm\bf H}_0.
\end{equation*}
The homology of $P(T)$ is given by
\begin{equation*}
{\rm H}_i(P(T)) = 
\begin{cases}
0 & {\rm for} \ i = 1 \ {\rm and} \ i=3, \\
{\rm H}_2 C = \pi_2T & {\rm for} \ i = 2, \\
\Gamma(\pi_2(T)) & {\rm for} \ i = 4.
\end{cases}
\end{equation*}

\begin{lemma}\label{betabar}
Given a $2$--type $T$, there is a chain map
\begin{equation*}
\overline \beta: \widehat C(T) \longrightarrow P(T)
\end{equation*}
inducing isomorphisms in homology in degree $\leq 4$. The map $\overline \beta$ is natural in $T$ up to homotopy, that is, a map $g: T \rightarrow T'$ between $2$--types yields a homotopy commutative diagram
\begin{equation*}
\xymatrix{
\widehat C(T) \ar[d]_{\overline \beta} \ar[rr]^{g_{\ast}} && \widehat C(T') \ar[d]^{\overline \beta} \\
P(T) \ar[rr]^{\overline{\overline g}_{\ast}} && P(T'),}
\end{equation*}
where $\overline{\overline g}_{\ast}$ is induced by a map $\overline{\overline g}: \partial_f \rightarrow \partial_{f'}$ associated with $g$.
\end{lemma}

For a proof of Lemma \ref{betabar}, we refer the reader to diagram (1.2) in Chapter V of \cite{Baues1}. In order to compute the fourth homology or cohomology of a $2$--type $T$ with coefficients, choose a pair $(\partial_f, B)$ representing $T$ and a free chain complex $C$ together with a weak equivalence of chain complexes
\begin{equation*}
C \stackrel{\sim}{\longrightarrow} P(\partial_f, B).
\end{equation*}
Then, for right $\Lambda$--modules $M$ and left $\Lambda$--modules $N$,
\begin{eqnarray*}
{\rm H}_4(T, M) & = & {\rm H}_4(C \otimes M), \\
{\rm H}^4(T, N) & = & {\rm H}^4({\rm Hom}_{\Lambda}(C, N)).
\end{eqnarray*}
This allows for the computation of ${\rm H}_4$ in terms of chain complexs only, as is the case for the computation of group homology in Section \ref{pd3section}. Of course, it is also possible to compute the homology of $T$ in terms of a spectral sequence associated with the fibration
\begin{equation*}
K(\pi_2(T), 2) \longrightarrow T \longrightarrow K(\pi_1(T),1).
\end{equation*}
However, in general, this yields non--trivial differentials, which may be related to the properties of the chain complex $P(\partial_f, B)$.

\section{Algebraic models of ${\rm PD}^4$--complexes}\label{algmodPD4}
Let $X$ be a $4$--dimensional ${\rm CW}$--complex and let
\begin{equation*}
p_2: X \longrightarrow P_2X = T
\end{equation*}
be the map to the $2$--type of $X$, as in (\ref{postnikov}). Then $p_2$ yields the chain map
\begin{equation*}
\beta: \widehat C(X) 
\stackrel{p_{2\ast}}{\longrightarrow} \widehat C(T) 
\stackrel{\overline \beta_{\ast}}{\longrightarrow} P(T) = P(\partial_f, B),
\end{equation*}
were $\partial_f$ is given by the attaching map of $2$--cells in $X$ and $B = {\rm im}(d_3: \widehat C_3(X) \rightarrow \widehat C_2(X))$. We call the chain map $\beta$ the \emph{cellular boundary invariant of $X$}.

\begin{lemma}\label{4realize}
Let $X$ and $X'$ be $4$--dimensional ${\rm CW}$--complexes. A chain map $\varphi: \widehat C(X) \rightarrow \widehat C(X')$ is realizable by a map $g: X \rightarrow X'$ in ${\rm\bf CW}_0$, that is, $\varphi = g_{\ast}$, if and only if the diagram
\begin{equation*}
\xymatrix{
\widehat C(X) \ar[rr]^{\varphi} \ar[d]_{\beta} && \widehat C(X') \ar[d]^{\beta'} \\
P(\partial_f, B) \ar[rr]^{\overline{\overline \varphi}} && P(\partial_{f'}, B')}
\end{equation*}
commutes up to homotopy. Here $\overline{\overline \varphi}: \partial_f \rightarrow \partial_{f'}$ is a map in ${\rm\bf cross}^=$ inducing $\varphi_{\leq 2}: \widehat C(X^2) \rightarrow \widehat C(X'^2)$ as in Lemma \ref{whitehead}.
\end{lemma}

\begin{proof}
By Lemma \ref{betabar}, the diagram
\begin{equation*}
\xymatrix{
\widehat C(X) \ar[rr]^{\varphi} \ar[d]_{p_{2\ast}} && \widehat C(X') \ar[d]^{p_{2\ast}} \\
\widehat C(T) \ar[rr]^{g_{\ast}} && \widehat C(T')}
\end{equation*}
is homotopy commutative, where $g$ is given by $q(\overline{\overline g})$ in ${\rm\bf Ho(cross)}$. Since $p_{2\ast}$ and $g_{\ast}$ are realizable, the obstruction $\mathcal O_{X, X'}(\varphi)$ vanishes.
\end{proof}

\begin{definition}
A $\beta$--${\rm PD}^4$--chain complex is a ${\rm PD}^4$--chain complex $((\pi, C), \omega, [C], \Delta)$ together with a totally free pre--crossed module $\partial_f$ inducing $d_2: C_2 \rightarrow C_1$ and a chain map
\begin{equation*}
\beta: C \longrightarrow P(\partial_f, B)
\end{equation*}
which is the identity in degree $\leq 2$. Here $B = {\rm im}(d_3: C_3 \rightarrow C_2)$, the diagram
\begin{equation*}
\xymatrix{
C \ar[rr]^{\Delta} \ar[d]_{\beta} && C \otimes C \ar[d]^{\beta^{\otimes}} \\
P(\partial_f, B) \ar[rr]^{\overline{\overline \Delta}_{\ast}} && P(\partial_{f\otimes f}, B^{\otimes})}
\end{equation*}
commutes up to homotopy and $\beta$ is the cellular boundary invariant $\beta_{\sigma}$ of a totally free quadratic chain complex $\sigma$ defined in V(1.8) of \cite{Baues1}. Further, $\beta^{\otimes}$ is the cellular boundary invariant of the quadratic chain complex $\sigma \otimes \sigma$ defined in Section IV 12 of \cite{Baues1}, and there is an explicit formula expressing $\beta^{\otimes}$ in terms of $\beta$, which we do not recall here. The function $f \otimes f$ is the attaching map of $2$--cells in the product $X^2 \times X^2$, where $X^2$ is given by $f$, and $B^{\otimes}$ is the image of $d_3$ in $C \otimes C$. The map $\overline{\overline \Delta}$ in ${\rm\bf cross}^=$ is chosen such that $\overline{\overline \Delta}$ induces $\Delta$ in degree $\leq 2$ as in Lemma \ref{4realize}.) Let ${\rm\bf PD}^4_{\ast,\beta}$ be the category whose objects are $\beta$--${\rm PD}^4$--chain complexes and whose morphisms are maps $\varphi$ in ${\rm\bf PD}^4_{\ast}$ such that the diagram
\begin{equation*}
\xymatrix{
C \ar[rr]^{\varphi} \ar[d]_{\beta} && C' \ar[d]^{\beta'} \\
P(\partial_f, B) \ar[rr]^{\overline{\overline \varphi}} && P(\partial_{f'}, B')}
\end{equation*}
is homotopy commutative, where $\overline{\overline \varphi}$ induces $\varphi_{\leq 2}$ as in Lemma \ref{4realize}.
\end{definition}

\begin{theorem}
The functor $\widehat C$ yields a functor
\begin{equation*}
\widehat C: {\rm\bf PD}^4 / \simeq \longrightarrow {\rm\bf PD}^4_{\ast,\beta} / \simeq
\end{equation*}
which reflects isomorphisms and is representative and full.
\end{theorem}

\begin{proof}
Since $C$ is $2$--realizable, there is a $4$--dimensional ${\rm CW}$--complex $X$ with $\widehat C(X) = C$ and cellular boundary invariant $\beta$. By Lemma \ref{4realize}, the diagonal $\Delta$ is realizable by a diagonal $X \rightarrow X \times X$, showing that $X$ is a ${\rm PD}^4$--complex. By Lemma \ref{4realize}, a map $\varphi$ is realizable by a map $X \rightarrow X'$.
\end{proof}

\begin{corollary}\label{PD4cor}
The functor $\widehat C$ induces a 1--1 correspondence between homotopy types of ${\rm PD}^4$--complexes and homotopy types of $\beta$--${\rm PD}^4$--chain complexes.
\end{corollary}

The functor $\tau$ in Section \ref{triplesection} yields the diagram of functors
\begin{equation}\label{taustar4}
\xymatrix{
{\rm\bf PD}^4_+ / \simeq \ar[rr]^{\widehat C} \ar[dr]_{\tau_+} && 
{\rm\bf PD}^4_{\ast +,\beta} / \simeq \ar[dl]^{\tau_{\ast}} \\
& {\rm\bf Trp}^4_+ &}
\end{equation}
where $\tau_+$ determines $\tau_{\ast}$ together with a natural isomorphism $\tau_{\ast} \widehat C \cong \tau_+$.

\begin{corollary}
The functor $\tau_{\ast}$ in (\ref{taustar4}) reflects isomorphisms and is full.
\end{corollary}

\section{Homotopy systems of order (k+1)}\label{homsys}
To investigate questions of realizability, we work in the category ${\rm\bf H}_c^{k+1}$ of homotopy systems of order $(k+1)$. Let ${\rm\bf CW}_0^k$ be the full subcategory of ${\rm\bf CW}_0$ consisting of $k$--dimensional CW--complexes. A $0$--homotopy $H$ in ${\rm\bf CW}_0$, denoted by $\simeq^0$, is a homotopy for which $H_t$ is cellular for each $t, 0 \leq t \leq 1$.

Let $k \geq 2$. A {\emph{homotopy system of order $(k+1)$}} is a triple $X = (C, f_{k+1}, X^k)$, where $X^k$ is an object in ${\rm\bf CW}_0^k$, $C$ is a chain complex of free $\pi_1(X^k)$--modules, which coincides with $\widehat C(X^k)$ in degree $\leq k$, and where $f_{k+1}$ is a homomorphism of left $\pi_1(X^k)$--modules such that
\begin{equation*}
\xymatrix{
C_{k+1} \ar[d]_d \ar[r]^{f_{k+1}} & \pi_k(X^k) \ar[d]^j \\
C_k & \pi_k( {X}^k,  {X}^{k-1}) \ar[l]_-{h_k}}
\end{equation*}
commutes. Here $d$ is the boundary in $C$,
\begin{equation*}
\xymatrix{
h_k: \pi_k({X}^k, {X}^{k-1}) \ar[r]^{p_{\ast}^{-1}}_{\cong} & \pi_k(\widehat{X}^k, \widehat{X}^{k-1}) \ar[r]^h_{\cong} & {\rm{H}}_k(\widehat{X}^k, \widehat{X}^{k-1}),}
\end{equation*}
given by the Hurewicz isomorphism $h$ and the inverse of the isomorphism on the relative homotopy groups induced by the universal covering $p: \widehat X \rightarrow X$. Moreover, $f_{k+1}$ satisfies the {\emph{cocycle condition}}
\begin{equation*}
f_{k+1} d (C_{k+2}) = 0.
\end{equation*}
For an object $X$ in ${\rm\bf CW}_0$, the triple $r(X) = (\widehat C(X), f_{k+1}, X^k)$ is a homotopy system of order $(k+1)$, where $X^k$ is the $k$--skeleton of $X$, and
\begin{equation*}
\xymatrix{
f_{k+1}: \widehat C_{k+1}(X) \cong \pi_{k+1}(X^{k+1}, X^k) \ar[r]^-{\partial} & \pi_k(X^k)}
\end{equation*}
is the attaching map of $(k+1)$--cells in $X$. A {\emph{morphism}} or {\emph{map}} between homotopy systems of order $(k+1)$ is a pair
\begin{equation*}
(\xi, \eta): (C, f_{k+1}, X^k) \rightarrow (C', g_{k+1}, Y^k),
\end{equation*}
where $\eta: X^k \rightarrow Y^k$ is a morphism in ${\rm\bf CW}_0 / \simeq^0$ and the $\pi_1(\eta)$--equivariant chain map $\xi: C \rightarrow C'$ coincides with $\widehat C_{\ast}(\eta)$ in degree $\leq k$ such that
\begin{equation*}
\xymatrix{
C_{k+1} \ar[d]^{f_{k+1}} \ar[r]^{\xi_{k+1}} & C_{k+1}' \ar[d]^{g_{k+1}} \\
\pi_k(X^k) \ar[r]_{\eta_{\ast}} & \pi_k(Y^k)}
\end{equation*}
commutes. We also write $\pi_1 X = \pi_1 (X^k)$ for an object $X = (C, f_{k+1}, X^k)$ in ${\rm\bf H}_{k+1}^c$.

To define the homotopy relation in ${\rm{\bf H}}_{k+1}^c$, we use the action (see ??? in \cite{Baues1})
\begin{equation}\label{indaction}
[X^k, Y]_{\varphi} \times {\widehat{\rm{H}}}^k(X^k, \varphi^{\ast} \pi_kY) \rightarrow [X^k, Y]_{\varphi}, \quad (F, \{\alpha\}) \mapsto F+\{\alpha\},
\end{equation}
where $[X^n, Y]_{\varphi}$ is the set of elements in $[X^n, Y]$ which induce $\varphi$ on the fundamental groups. Two morphisms
\begin{equation*}
(\xi, \eta), (\xi', \eta'): (C, f_{k+1}, X^k) \rightarrow (C', g_{k+1}, Y^k)
\end{equation*}
are {\emph{homotopy equivalent}} in ${\rm\bf H}_{k+1}^c$ if $\pi_1(\eta) = \pi_1(\eta') = \varphi$ and if there are $\varphi$--equivariant  homomorphisms $\alpha_{j+1}: C_j \rightarrow C'_{j+1}$ for $j \geq k$ such that
\begin{equation*}
\{\eta\} + g_{k+1} \alpha_{k+1} = \{\eta'\} \quad {\rm{and}} \\
\end{equation*}
\begin{equation*}
\xi'_i - \xi_i = \alpha_i d + d \alpha_{i+1}, \quad  i \geq k+1,
\end{equation*}
where $\{\eta\}$ denotes the homotopy class of $\eta$ in $[X^k, Y^k]$ and $+$ is the action (\ref{indaction}).

Given homotopy systems $X = (C, f_{k+1}, X^k)$ and $Y = (C', g_{k+1}, Y^k)$, consider
\begin{equation*}
X \otimes Y = (C \otimes_{\mathbb Z} C', h_{k+1}, (X^k \times Y^k)^k), 
\end{equation*}
where we choose CW--complexes $X^{k+1}$ and $Y^{k+1}$ with attaching maps $f_{k+1}$ and $g_{k+1}$, respectively, and $h_{k+1}$ is given by the attaching maps of $(k+1)$--cells in $X^{k+1} \times Y^{k+1}$. Then $X \otimes Y$ is a homotopy system of order $(k+1)$, and 
\begin{equation*}
\otimes: {\rm\bf H}_{k+1}^c \times {\rm\bf H}_{k+1}^c \rightarrow {\rm\bf H}_{k+1}^c
\end{equation*}
is a bi--functor, called the \emph{tensor product of homotopy systems}. The projections  $p_1 :  X \otimes Y \rightarrow X$ and $p_2:  X \otimes Y \rightarrow Y$ in ${\rm\bf H}_{k+1}^c$ are given by the projections of the tensor product and the product of $\rm CW$--complexes. Similarly, we obtain the inclusions $\iota_1:  X \rightarrow  X \otimes  Y$ and $\iota_2:  Y \rightarrow  X \otimes  Y$. Then $p_1  \iota_1 = {\rm id}_X$ and $p_2  \iota_2 = {\rm id}_Y$, while $p_1  \iota_2$ and $p_2 \iota_1$ yield the trivial maps.

There are functors
\begin{equation}\label{functors}
\xymatrix{
{\rm\bf CW}_0 \ar[r]^-r & {\rm\bf H}_{k+1}^c \ar[r]^-{\lambda} & {\rm\bf H}_{k}^c \ar[r]^C & {\rm\bf H}_0}
\end{equation}
for $k \geq 3$, with $r(X) = (\widehat C(X), f_{k+1}, X^k)$ such that $r = \lambda r$. We write $\lambda X = \overline X$ for objects $X$ in ${\rm\bf H}_{k+1}^c$. As $\overline{X \otimes Y} = \lambda(X \otimes Y) = \overline X \otimes \overline Y$, the functor $\lambda$, and also $r$ and $C$, is a monoidal functor between monoidal categories. There is a homotopy relation defined on the category ${\rm\bf H}_{k+1}^c$ such that these functors induce functors between homotopy categories
\begin{equation*}
\xymatrix{
{\rm\bf CW}_0 / \simeq \ar[r]^-r & {\rm{\bf H}}_{k+1}^c / \simeq \ar[r]^-{\lambda} & {\rm{\bf H}}_{k}^c / \simeq \ar[r]^C & {\rm\bf H}_0/\simeq.}
\end{equation*}
For $k \geq 3$, Whitehead's functor $\Gamma_k$ factors through the functor $r: {\rm\bf CW} \rightarrow {\rm\bf H}_k^c$, so that the cohomology $\widehat {\rm H}_m (\overline X, \varphi^{\ast} \Gamma_k (\overline Y)) =  {\rm H}^m (C, \varphi^{\ast} \Gamma_k (\overline Y))$ is defined, where $\varphi: \pi_1 \overline X \rightarrow \pi_1 \overline Y$ and $\overline X$ and $\overline Y$ are objects in ${\rm\bf H}_k^c$.

To describe the obstruction to realizing a map $f = (\xi, \eta): \overline X \rightarrow \overline Y$ in ${\rm{\bf H}}_k^c$, where $\overline X = \lambda X$ and $\overline Y = \lambda Y$, by a map $X \rightarrow Y$ in ${\rm{\bf H}}_{k+1}^c$ for objects $X = (C, f_{k+1}, X^k)$ and $Y = (C', g_{k+1}, Y^k)$, choose a map $F: X^k \rightarrow Y^k$ in ${\rm\bf CW} / {\simeq}^0$ extending $\eta: X^{k-1} \rightarrow Y^{k-1}$ and for which $\widehat C_{\ast} F$ coincides with $\xi$ in degree $\leq n$. Then 
\begin{equation*}
\xymatrix{
C_{k+1} \ar[d]^{f_{k+1}} \ar[r]^{\xi_{k+1}} & C_{k+1}' \ar[d]^{g_{k+1}} \\
\pi_k(X^k) \ar[r]^{F_{\ast}} & \pi_k(Y^k)}
\end{equation*}
need not commute and the difference $\mathcal O(F) = -g_{k+1} \xi_{k+1} + F_{\ast} f_{k+1}$ is a cocycle in ${\rm{Hom}}_{\varphi}(C_{k+1}, \Gamma_k(\overline Y))$. Theorem II 3.3 in \cite{Baues1} implies 
\begin{proposition}\label{realize}
The map $f = (\xi, \eta): \overline X \rightarrow \overline Y$ in ${\rm{\bf H}}_k^c$ can be realized by a map $f_0 = (\xi, \eta_0): X \rightarrow Y$ in ${\rm\bf H}_{k+1}^c$ if and only if $\mathcal O_{X, Y}(f) = \{ \mathcal O(F) \} \in \widehat{\rm H}^{k+1}(\overline X, \varphi^{\ast} \Gamma_k \overline Y)$ vanishes. The obstruction $\mathcal O$ is a derivation, that is, for $f: \overline X \rightarrow \overline Y$ and $g: \overline Y \rightarrow \overline Z$,
\begin{equation}
\mathcal O_{X, Z}(gf) = g_{\ast} \mathcal O_{X, Y}(f) + f^{\ast} \mathcal O_{Y, Z}(g),
\end{equation}
and $\mathcal O_{X, Y}(f)$ depends on the homotopy class of $f$ only.
\end{proposition}
Denoting the set of morphisms $X \rightarrow Y$ in ${\rm\bf H}_{k+1}^c / \simeq$ by $[X, Y]$, and the subset of morphisms inducing $\varphi$ on the fundamental groups by $[X, Y]_{\varphi} \subseteq [X, Y]$, there is a group action
\begin{equation*}
[X, Y]_{\varphi} \times \widehat{\rm H}^k(\overline X, \varphi^{\ast} \Gamma_k \overline Y) \stackrel{+}{\longrightarrow} [X, Y]_{\varphi},
\end{equation*} 
where $\overline X = \lambda X$ and $\overline Y = \lambda Y$. Theorem II 3.3 in \cite{Baues1} implies 
\begin{proposition}\label{actdist}
Given morphisms $f_0, f_0' \in [X, Y]_{\varphi}$, then $\lambda f_0 = \lambda f_0' = f$ if and only if there is an $\alpha \in \widehat{\rm{H}}^k(\overline X, \varphi^{\ast} \Gamma_k \overline Y)$ with $f_0' = f_0 + \alpha$. In other words, $\widehat{\rm{H}}^k(\overline X, \varphi^{\ast} \Gamma_k \overline Y)$ acts transitively on the set of realizations of $f$. Further, the action satisfies the linear distributivity law
\begin{equation}\label{distributivity}
(f_0 + \alpha )(g_0 + \beta) = f_0 g_0 + f_{\ast} \beta + g^{\ast} \alpha.
\end{equation}
\end{proposition}
For the functor $\lambda$ in (\ref{functors}), Theorem II 3.3 in \cite{Baues1} implies 
\begin{proposition}\label{objrealize}
For all objects $X$ in ${\rm\bf H}_{k+1}^c$ and for all $\alpha \in \widehat{\rm H}^{k+1}(\overline X, \Gamma_k \overline X)$, there is an object $X'$ in ${\rm\bf H}_{k+1}^c$ with $\lambda(X') = \lambda(X) = \overline X$ and $\mathcal O_{X, X'}({\rm id}_{\overline X}) = \alpha$. We then write $X' = X + \alpha$. 

Now let $Y$ be an object in ${\rm\bf H}_k^c$. Then the group $\widehat{\rm H}^{k+1}(Y, \Gamma_k Y)$ acts transitively and effectively on ${\rm Real}_{\lambda}(Y)$ via $+$, provided ${\rm Real}_{\lambda}(Y)$ is non--empty. Moreover, ${\rm Real}_{\lambda}(Y)$ is non--empty if and only if an obstruction $\mathcal O(Y) \in \widehat{\rm H}^{k+2}(Y, \Gamma_k Y)$ vanishes.
\end{proposition}
For objects $X$ and $Y$ in ${\rm\bf H}_{k+1}^c$ and a morphism $f: \overline X \rightarrow \overline Y$ in ${\rm\bf H}_k^c$, Propositions \ref{realize} and \ref{objrealize} yield
\begin{equation}
\mathcal O_{X + \alpha, Y + \beta}(f) =  \mathcal O_{X, Y}(f) - f_{\ast} \alpha + f^{\ast} \beta
\end{equation}
for all $\alpha \in \widehat{\rm H}^{k+1}(\overline X, \Gamma_k \overline X)$ and $\beta \in \widehat{\rm H}^{k+1}(\overline Y, \Gamma_k \overline Y)$. Given another object $Z$ in ${\rm\bf H}_{k+1}^c$ with $\lambda Z = \overline Z$,
\begin{eqnarray}
\mathcal O_{X \otimes Z, Y \otimes Z}(f \otimes {\rm id}_{\overline Z}) 
& = &  \overline \iota_{1\ast} \overline p_1^{\ast} \mathcal O_{X, Y}(f), \label{tens1} \\
\mathcal O_{Z \otimes X, Z \otimes Y}({\rm id}_{\overline Z} \otimes f) 
& = & \overline \iota_{2\ast} \overline p_2^{\ast} \mathcal O_{X, Y}(f), \label{tens2}
\end{eqnarray}
where $\overline \iota_1: \overline X \rightarrow \overline X \otimes \overline Z$ and $\overline p_1: \overline X \otimes \overline Z \rightarrow \overline X$ are the inclusion of and projection onto the first factor and $\overline \iota_2$ and $\overline p_2$ are defined  analogously. We obtain
\begin{equation}\label{tensobjact}
(X + \alpha) \otimes (Y + \beta) = (X \otimes Y) + \overline \iota_{1\ast} \overline p_1^{\ast} \alpha + \overline \iota_{2\ast} \overline p_2^{\ast} \beta.
\end{equation}

\section{Obstructions to the diagonal}\label{diagobs}
Let $k \geq 2$. A \emph{diagonal} on $X = (C, f_{k+1}, X^k)$ in ${\rm\bf H}_{k+1}^c$ is a morphism, $\Delta: X \rightarrow X \otimes X$, such that, for $i = 1, 2$, the diagram
\begin{equation}
\xymatrix{
X \ar[rr]^-{\Delta} \ar[drr]_{\rm id} && X \otimes X \ar[d]^{p_i} \\
&& X}
\end{equation}
commutes up to homotopy in ${\rm\bf H}_{k+1}^c$. 
Applying the functor $r: {\rm\bf CW}_0 \rightarrow {\rm\bf H}_{k}^c$ to a diagonal $\Delta: X \rightarrow X \times X$ in ${\rm\bf CW}_0$, we obtain the diagonal $r(\Delta): r(X) \rightarrow r(X) \otimes r(X)$
in ${\rm\bf H}_k^c$. 
\begin{lemma}\label{diagreal0}
Let $X$ be an object in ${\rm\bf H}_{k+1}^c$. Then every $\lambda$--realizable diagonal $\overline \Delta: \overline X = \lambda X \rightarrow \overline X \otimes \overline X$ in ${\rm\bf H}_k^c/\simeq$ has a $\lambda$--realization $\Delta: X \rightarrow X \otimes X$ in ${\rm\bf H}_{k+1}^c/\simeq$ which is a diagonal in ${\rm\bf H}_{k+1}^c$.
\end{lemma}
\begin{proof}
Suppose $\Delta': X \rightarrow X \otimes X$ is a $\lambda$--realization of $\overline \Delta$ in ${\rm\bf H}_{k+1}^c$. The projection $p_{\ell}: X \rightarrow X \otimes X$ realizes the projection $\overline{p}_{\ell}: \overline X \rightarrow \overline X \otimes \overline X$ and hence $p_{\ell} \Delta'$ realizes $\overline p_{\ell} \overline \Delta$ for $\ell = 1, 2$. Now the identity on $X$ realizes the identity on $\overline X$ and $\overline p_{\ell} \Delta$ is homotopic to the identity on $\overline X$ by assumption. Hence $p_{\ell} \Delta'$ and the identity on $X$ realize the same homotopy class of maps for $\ell = 1, 2$. As the group ${\widehat{\rm H}}^k(\overline{X}, \Gamma_k \overline X)$ acts transitively on the set of realizations of this homotopy class by Proposition \ref{actdist}, there are elements $\alpha_{\ell} \in {\widehat{\rm H}}^k(\overline X, \Gamma_k \overline X)$ such that
\begin{equation*}
\{p_{\ell} \Delta' \} + \alpha_{\ell} = \{ {\rm id}_X \} \quad {\rm{for}} \ {\ell} = 1,2,
\end{equation*}
where $\{f\}$ denotes the homotopy class of the morphism $f$ in ${\rm\bf H}_{k+1}^c$. We put
\begin{equation*}
\{\Delta\} = \{\Delta'\} + \iota_1 \alpha_1 + \iota_2 \alpha_2.
\end{equation*}
By Proposition \ref{actdist},
\begin{eqnarray*}
\{p_{\ell} \Delta \} & = & \{p_{\ell}\}(\{\Delta'\} + \iota_1 \alpha_1 + \iota_2 \alpha_2) \\
& = & \{p_{\ell} \Delta'\} + \overline p_{\ell \ast} \iota_1 \alpha_1 + \overline p_{\ell \ast} \iota_2 \alpha_2 \\
& = & \{p_{\ell} \Delta' \} +  \alpha_{\ell}  
=  \{ {\rm id}_X \}.
\end{eqnarray*}
\end{proof}
\begin{lemma}\label{diagreal1}
For $X$ in ${\rm\bf H}_{k+1}^c$, let $\Delta_{\overline X}: \overline X \rightarrow \overline X \otimes \overline X$ be a diagonal on $\overline X = \lambda X$ in ${\rm\bf H}_k^c$. Then we obtain, in ${\rm H}^{k+1}(\overline X, \Gamma_k(\overline X \otimes \overline X)$,
\begin{enumerate}
\item $O_{X, X \otimes X}(\Delta_{\overline X}) \in {\rm ker} \ \overline p_{i\ast}$ for $i = 1, 2$, 
\item $\mathcal O_{X, X \otimes X}(\Delta_{\overline X}) \in {\rm ker}({\rm id}_{\overline X \ast} - T_{\ast})_{\ast}$ if $\Delta_{\overline X}$ is homotopy commutative and
\item $\mathcal O_{X, X \otimes X}(\Delta_{\overline X}) \in {\rm ker}\big(\overline \iota_{1,2 \ast} - \overline \iota_{2,3 \ast} + (\Delta_{\overline X} \otimes {\rm id}_{\overline X})_{\ast} - ({\rm id}_{\overline X} \otimes \Delta_{\overline X})_{\ast}\big)_{\ast}$ if $\Delta_{\overline X}$ is homotopy associative. \label{homass}
\end{enumerate}
\end{lemma}
\begin{proof}
By definition, $\overline p_i \Delta_{\overline X} \simeq {\rm id}_{\overline X}$ for $i = 1, 2$. As the identity on $\overline X$ is realized by the identity on $X$ and $\overline p_i: \overline X \otimes \overline X \rightarrow \overline X$ is realized by $p_i: X \otimes X \rightarrow X$, Proposition \ref{realize} implies $\mathcal O_{X, X \otimes X}({\rm id}_{\overline X}) = 0$ and $\mathcal O_{X \otimes X, X}(\overline p_i) = 0$. Since $\mathcal O$ is a derivation, we obtain 
\begin{equation*}
0 = \mathcal O_{X, X}(\overline p_i \Delta_{\overline X}) = \overline p_{i\ast} \mathcal O_{X, X \otimes X}(\Delta_{\overline X}) + \Delta_{\overline X}^{\ast} \mathcal O_{X \otimes X, X}(\overline p_i) = \overline p_{i\ast} \mathcal O_{X, X \otimes X}(\Delta_{\overline X}),
\end{equation*}
and hence $O_{X, X \otimes X}(\Delta_{\overline X}) \in {\rm ker} \ \overline p_{i\ast}$ for $i = 1, 2$. If $\Delta_{\overline X}$ is homotopy commutative, then 
\begin{equation*}
\mathcal O_{X, X \otimes X}(\Delta_{\overline X}) = \mathcal O_{X, X \otimes X}(T \Delta_{\overline X}) = T_{\ast} \mathcal O_{X, X \otimes X}(\Delta_{\overline X}),
\end{equation*}
since $\mathcal O_{X \otimes X, X \otimes X}(T) = 0$, as $T$ is $\lambda$--realizable. Hence $\mathcal O_{X, X \otimes X}(\Delta_{\overline X}) \in {\rm ker}({\rm id}_{\overline X \ast} - T_{\ast})_{\ast}$. For $1 \leq k < \ell, \leq 3$, let $\iota_{k,\ell}: X \otimes X \rightarrow X \otimes X \otimes X$ denote the inclusion of the $k$-th and $\ell$-th factors and suppose $\Delta_{\overline X}$ is a homotopy commutative diagonal in ${\rm\bf H}_k^c$. Then $\mathcal O_{X, X \otimes X \otimes X}((\Delta_{\overline X} \otimes {\rm id}_{\overline X}) \Delta_{\overline X}) = \mathcal O_{X, X \otimes X \otimes X}(({\rm id}_{\overline X} \otimes \Delta_{\overline X}) \Delta_{\overline X})$, as the obstruction depends on the homotopy class of a morphism only, and
\begin{eqnarray*}
\mathcal O_{X, X \otimes X \otimes X}(\Delta_{\overline X} \otimes {\rm id}_{\overline X}) 
& = & \overline \iota_{1,2 \ast} \overline p_1^{\ast} \mathcal O_{X, X \otimes X}(\Delta_{\overline X}) \\
\mathcal O_{X, X \otimes X \otimes X}({\rm id}_{\overline X} \otimes \Delta_{\overline X}) 
& = & \overline \iota_{2,3 \ast} \overline p_2^{\ast} \mathcal O_{X, X \otimes X}(\Delta_{\overline X},
\end{eqnarray*}
by (\ref{tens1}) and (\ref{tens2}). Omitting the objects in the notation for the obstruction, we obtain
\begin{eqnarray*}
\mathcal O((\Delta_{\overline X} \otimes {\rm id}_{\overline X}) \Delta_{\overline X})
& = &  \Delta_{\overline X}^\ast \mathcal O(\Delta_{\overline X} \otimes {\rm id}_{\overline X}) 
+ (\Delta_{\overline X} \otimes {\rm id}_{\overline X})_{\ast} \mathcal O(\Delta_{\overline X}) \\
& = & \Delta_{\overline X}^\ast \overline \iota_{1,2 \ast} \overline p_1^{\ast} \mathcal O(\Delta_{\overline X}) + (\Delta_{\overline X} \otimes {\rm id}_{\overline X})_{\ast} \mathcal O(\Delta_{\overline X}) \\
& = &  \overline \iota_{1,2 \ast} (\overline p_1\Delta_{\overline X})^{\ast} \mathcal O(\Delta_{\overline X}) + (\Delta_{\overline X} \otimes {\rm id}_{\overline X})_{\ast} \mathcal O(\Delta_{\overline X}) \\
& = & \overline \iota_{1,2 \ast} \mathcal O(\Delta_{\overline X})
+ (\Delta_{\overline X} \otimes {\rm id}_{\overline X})_{\ast} \mathcal O(\Delta_{\overline X}).
\end{eqnarray*}
Similarly, we obtain
\begin{equation*}
\mathcal O(({\rm id}_{\overline X} \otimes \Delta_{\overline X}) \Delta_{\overline X})
= \overline \iota_{2,3 \ast} \mathcal O(\Delta_{\overline X})
+ (\Delta_{\overline X} \otimes {\rm id}_{\overline X})_{\ast} \mathcal O(\Delta_{\overline X}),
\end{equation*}
which proves (\ref{homass}).
\end{proof}
\begin{question}
Given a $\lambda$--realizable object $\overline X$ with a diagonal $\Delta_{\overline X}: \overline X \rightarrow \overline X \otimes \overline X$ in ${\rm\bf H}_k^c$, is there an object $X$ with $\lambda X = \overline X$ and a diagonal $\Delta_X: X \rightarrow X \otimes X$ in ${\rm\bf H}_{k+1}^c$ such that $\lambda \Delta_X = \Delta_{\overline X}$?
\end{question}
Let $X$ in ${\rm\bf H}_{k+1}^c$ be a $\lambda$--realization of $\overline X$. By Proposition \ref{objrealize}, any $\lambda$--realization $X'$ of $\overline X$ is of the form $X' = X + \alpha$ for some $\alpha \in \widehat{\rm H}^{k+1}(\overline X, \Gamma_k \overline X)$. By (\ref{tensobjact}), $X' \otimes X' = (X \otimes X) + \overline \iota_{1\ast} \overline p_1^{\ast} \alpha + \overline \iota_{2\ast} \overline p_2^{\ast} \alpha$ and as the obstruction $\mathcal O$ is a derivation, we obtain
\begin{eqnarray*}
\mathcal O_{X', X' \otimes X'}(\Delta_{\overline X})
& = & \mathcal O_{X + \alpha, (X + \alpha) \otimes (X + \alpha)}(\Delta_{\overline X})\\
& = & \mathcal O_{X, X \otimes X}(\Delta_{\overline X}) - \Delta_{\overline X \ast} \alpha + \Delta_{\overline X}^{\ast} (\overline \iota_{1\ast} \overline p_1^{\ast} \alpha + \iota_{2\ast} \overline p_2^{\ast} \alpha)\\
& = & \mathcal O_{X, X \otimes X}(\Delta_{\overline X}) - (\Delta_{\overline X \ast} - \overline \iota_{1 \ast} - \overline \iota_{2 \ast}) \alpha,
\end{eqnarray*}
since $\Delta^{\ast} \overline \iota_{i\ast} \overline p_i^{\ast} = \overline \iota_{i\ast} (\overline p_i \Delta)^{\ast} = \overline \iota_{i\ast}$, for $i = 1, 2$.

\begin{lemma}\label{obsact}
For $X$ in ${\rm\bf H}_{k+1}^c$, let $\Delta_{\overline X}: \overline X \rightarrow \overline X \otimes \overline X$ be a diagonal on $\overline X = \lambda X$ in ${\rm\bf H}_k^c$ and let $X' = X + \alpha$ for some $\alpha \in \widehat{\rm H}^{k+1}(\overline X, \Gamma_k \overline X)$. Then we obtain, in ${\rm H}^{k+1}(\overline X, \Gamma_k(\overline X \otimes \overline X)$,
\begin{equation*}
\mathcal O_{X', X' \otimes X'}(\Delta_{\overline X}) = \mathcal O_{X, X \otimes X}(\Delta_{\overline X}) - (\Delta_{\overline X \ast} - \overline \iota_{1 \ast} - \overline \iota_{2 \ast}) \alpha.
\end{equation*}
\end{lemma}

\section{$\rm PD^n$--homotopy systems}\label{pdhomsys}
A \emph{$\rm PD^n$--homotopy system} $X = (X, \omega_X, [X], \Delta_X)$ of order $(k+1)$ consists of an object $X = (C, f_{k+1}, X^k)$ in ${\rm\bf H}_{k+1}^c$, a group homomorphism $\omega_X: \pi_1 X \rightarrow \mathbb Z / 2 \mathbb Z$, a fundamental class $[X] \in {\rm H}_n(C, \mathbb Z^{\omega})$ and a diagonal $\Delta: X \rightarrow X \otimes X$ in ${\rm\bf H}_{k+1}^c$ such that $(C, \omega_X, [X], \Delta_X)$ is a ${\rm PD}^n$--chain complex. A map $f:  (X, \omega_X, [X], \Delta_X) \rightarrow (Y, \omega_Y, [Y], \Delta_Y)$ of $\rm PD^n$--homotopy systems of order $(k+1)$ is a morphism in ${\rm\bf H}_{k+1}^c$ such that $\omega_X = \omega_Y \pi_1(f)$ and $(f \otimes f) \Delta_X \simeq \Delta_Y f$,
and we thus obtain the category ${\rm\bf PD}^n_{[k+1]}$ of $\rm PD^n$--homotopy systems of order $(k+1)$. Homotopies in ${\rm\bf PD}^n_{[k+1]}$ are homotopies in ${\rm\bf H}_{k+1}^c$, and restricting the functors in (\ref{functors}), we obtain, for $k \geq 3$, the functors
\begin{equation}
\xymatrix{
{\rm\bf PD}^n \ar[r]^-r & {\rm\bf PD}^n_{[k+1]} \ar[r]^-{\lambda} & {\rm\bf PD}_{[k]}^n \ar[r]^C & {\rm\bf PD}^n_{\ast}.}
\end{equation}
These functors induce functors between homotopy categories
\begin{equation*}
\xymatrix{
{\rm\bf PD}^n / \simeq \ar[r]^-r & {\rm\bf PD}^n_{[k+1]} / \simeq \ar[r]^-{\lambda} & {\rm\bf PD}^n_{[k]} / \simeq \ar[r]^C & {\rm\bf PD}^n_{\ast}/\simeq.}
\end{equation*}
\begin{theorem}\label{pdnfuncequ}
The functor $C: {\rm\bf PD}_{[3]}^n/\simeq \ \longrightarrow  {\rm\bf PD}^n_{\ast}/\simeq$ is an equivalence of categories for $n \geq 3$.
\end{theorem}
\begin{proof}
The functor $C$ is full and faithful by Theorem III 2.9 and Theorem III 2.12 in \cite{Baues1}. By Lemma \ref{2real}, every ${\rm PD}^n$--chain complex, $\overline X = (D, \omega, [D], \overline\Delta)$, in ${\rm\bf{PD}}^n_{\ast}$ is $2$--realizable, that is, there is an object $X^2$ in ${\rm CW}^2_0$ such that $\widehat C(X^2) = D_{\leq 2}$, and we obtain the object $X = (D, f_3, X^2)$ in ${\rm\bf H}_3^c$. As $C$ is monoidal, full and faithful, the diagonal $\overline\Delta$ on $\overline X$ is realized by a diagonal $\Delta$ on $X$ and hence $(X, \omega, [D], \Delta)$ is an object in ${\rm\bf PD}_{[3]}^n$ with $C(X) = \overline X$.
\end{proof}
\begin{theorem}\label{pdnfuncdet}
For $n \geq 3$, the functor $r: {\rm\bf PD}^n/\simeq \ \longrightarrow {\rm\bf PD}^n_{[n]}/\simeq$ reflects isomorphisms, is representative and full.
\end{theorem}
\begin{proof}
That $r$ reflects isomorphisms follows from Whitehead's Theorem. 

Poincar{\'e} duality implies $\widehat {\rm H}^{n+1}(Y, \Gamma_n Y) = \widehat {\rm H}^{n+2}(Y, \Gamma_n Y) = 0$, for every object $Y = (Y, \omega_Y, [Y], \Delta_Y)$ in ${\rm\bf PD}^n_{[n]}$. Hence, by Proposition \ref{objrealize}, $Y = \lambda(X)$ for some object $X$ in ${\rm\bf H}_{n+1}^c$, and, by Proposition \ref{realize}, the diagonal $\Delta_Y$ is $\lambda$--realizable. Thus Lemma \ref{diagreal0} guarantees the existence of a diagonal $\Delta_X: X \rightarrow X \otimes X$ in ${\rm\bf H}_{n+1}^c$ with $\lambda \Delta_X = \Delta_Y$. The homomorphism $\omega_Y$ and the fundamental class $[Y]$ determine a homomorphism $\omega_X: \pi_1 X \rightarrow \mathbb Z / 2 \mathbb Z$ and a fundamental class $[X] \in {\rm H}_n(C, \mathbb Z^{\omega})$, such that $X = (X, \omega_X, [X], \Delta_X)$ is an object in ${\rm\bf PD}^n_{[n+1]}$. Inductively, we obtain an object $(X_k, \omega_{X_k}, [X_k], \Delta_{X_k})$ realizing $(Y, \omega_Y, [Y], \Delta_Y)$ in ${\rm\bf PD}^n_{[k]}$ for $k > n$, and in the limit an object $X = (X, \omega_X, [X], \Delta_X)$ in ${\rm\bf PD}^n$ with $r(x) = Y$. 

Proposition \ref{realize} together with the fact that, by Poincar{\'e} duality, $\widehat {\rm H}^k(X, B) = 0$ for $k > n$ and every $\Lambda$--module $B$, implies that $r$ is full.
\end{proof}

Max--Planck--Institut f{\"u}r Mathematik 

Vivatsgasse 7

D--53111 Bonn, Germany

\email{baues@mpim-bonn.mpg.de} 

\email{bleile@mpim-bonn.mpg.de}

\medskip

The second author's home institution is

School of Science and Technology

University of New England 

NSW 2351, Australia

\email{bbleile@une.edu.au}

\end{document}